\documentclass[a4paper]{jpconf}
\usepackage{graphicx}
\usepackage{dcolumn}
\usepackage{bm}

\graphicspath{{../Graphs/}}

\usepackage{xcolor}
\usepackage{amsthm}
\usepackage{amsmath}
\usepackage{amssymb}
\usepackage{mathrsfs}
\usepackage{mathtools}

\usepackage{empheq}

\usepackage{dcolumn} 
\newcolumntype{d}[1]{D{.}{.}{#1}}

\usepackage{booktabs} 

\usepackage[labelformat=simple]{subcaption}
\usepackage{fancyhdr}
\usepackage{totpages}
\usepackage{txfonts}
\usepackage{hyperref}

\theoremstyle{definition}
\newtheorem{definition}{Definition}

\newtheorem{proposition}{Proposition}

\numberwithin{equation}{section}
\numberwithin{definition}{section}
\numberwithin{theorem}{section}
\numberwithin{remark}{section}
\numberwithin{proposition}{section}
\numberwithin{corollary}{section}

\newcommand{\dif}{\text{d}}

\DeclareMathOperator{\sech}{sech}
\DeclareMathOperator{\csch}{csch}

\newcommand{\ri}{\text{i}}
\newcommand{\rd}{\text{d}}

\newcommand{\D}{\displaystyle}

\allowdisplaybreaks

\begin{document}
\title{Non-self-adjoint sixth-order eigenvalue problems arising from clamped elastic thin films on closed domains}

\author{N C Papanicolaou$^1$ and I C Christov$^{2}$}

\address{$^1$ Department of Computer Science, University of Nicosia, 46 Makedonitissas Avenue, CY-2417 Nicosia, Cyprus}
\address{$^2$ School of Mechanical Engineering, Purdue University, West Lafayette, Indiana 47907, USA}

\ead{papanicolaou.n@unic.ac.cy,christov@purdue.edu}

\begin{abstract}
	Sixth-order boundary value problems (BVPs) arise in thin-film flows with a surface that has elastic bending resistance. We consider the case in which the elastic interface is clamped at the lateral walls of a closed trough and thus encloses a finite amount of fluid. For a slender film undergoing infinitesimal deformations, the displacement of the elastic surface from its initial equilibrium position obeys a sixth-order (in space) initial boundary value problem (IBVP). To solve this IBVP, we construct a set of odd and even eigenfunctions that intrinsically satisfy the boundary conditions (BCs) of the original IBVP. These eigenfunctions are the solutions of a non-self-adjoint sixth-order eigenvalue problem (EVP). To use the eigenfunctions for series expansions, we also construct and solve the adjoint EVP, leading to another set of even and odd eigenfunctions, which are orthogonal to the original set (biorthogonal). The eigenvalues of the adjoint EVP are the same as those of the original EVP, and we find accurate asymptotic formulas for them. Next, employing the biorthogonal sets of eigenfunctions, a Petrov--Galerkin spectral method for sixth-order problems is proposed, which can also handle lower-order terms in the IBVP. The proposed method is tested on two model sixth-order BVPs, which admit exact solutions. We explicitly derive all the necessary formulas for expanding the quantities that appear in the model problems into the set(s) of eigenfunctions. For both model problems, we find that the approximate Petrov--Galerkin spectral solution is in excellent agreement with the exact solution. The convergence rate of the spectral series is rapid, exceeding the expected sixth-order algebraic rate.
\end{abstract}




\section{Introduction}\label{sec:Intro}

We study the initial-boundary-value problem (IBVP) governing the small deformations of a clamped thin film with elastic surface resistance over a closed trough (section~\ref{sec:Target}), extending the investigation of Gabay \textit{et al}.~\cite{Gabay2023} on pinned films with surface tension in closed troughs. The corresponding sixth-order eigenvalue problem is \emph{not} self-adjoint. Building on our earlier works on self-adjoint sixth-order IBVPs \cite{NectarIvan2023,NectarIvan2024}, we construct the eigenfunctions and the \emph{adjoint} eigenfunctions (section~\ref{sec:EVP}). We use the latter in a Petrov--Galerkin framework (section~\ref{sec:Galerkin}) and apply this approach to two elementary model problems (section~\ref{sec:model_problems}) to demonstrate the order of convergence of the spectral expansion.

\section{Problem formulation}\label{sec:Target}
 
\subsection{Clamped thin film with bending resistance over a closed trough}

Consider a thin fluid film of equilibrium height $h_0$ and dynamic height $h(x,t)$, as shown in figure~\ref{fig:elastic_film_shematic}. The film's surface has elastic bending resistance $B$ but no inertia (the surface is idealized as a massless interface). Likewise, it is assumed that out-of-plane bending is the dominant elastic force on the interface, rather than in-plane tension, which can also be considered \cite{Hewitt2015,Tulchinsky2019,Boyko2019,Boyko2020}. The fluid is confined to a closed trough of width $2\ell$. Thus, fluid cannot leave or enter through the lateral boundaries at $x=\pm\ell$. The fluid film's surface is also considered clamped to its equilibrium height at the lateral boundaries, which for an elastic interface requires that $h|_{x=\pm\ell}=h_0$ \emph{and} $(\partial h/\partial x)|_{x=\pm\ell}=0$ $\forall t\ge 0$. The fluid is considered incompressible and Newtonian with (constant) density $\rho_f$ and (constant) dynamic viscosity $\mu_f$.  Gravity is the only {body} force. 

\begin{figure}[t]
	\begin{center}
		\includegraphics[width=.9\textwidth]{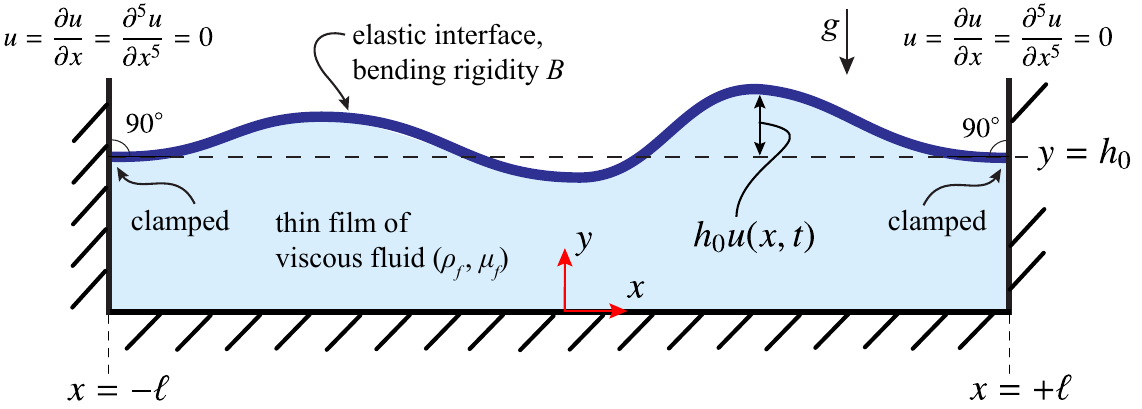}
	\end{center}
	\caption{Schematic illustration of the problem of a clamped elastic thin film on a closed domain. The elastic interface (with only out-of-plane bending rigidity and negligible mass) sits atop a viscous fluid of equilibrium height $h_0$ over a closed trough (no fluid flux through the lateral boundaries $x=\pm\ell$). Gravity is oriented in the $-y$ direction. When the interface is perturbed by an infinitesimal (dimensionless) displacement $u(x,t)$, the competition between the flow generated underneath it, its resistance to bending, and gravity sets the dynamics of leveling back to equilibrium, $u\to0$.}
	\label{fig:elastic_film_shematic}
\end{figure}

\subsection{Governing equations}

As shown in our previous works \cite{NectarIvan2023,NectarIvan2024,NectarIvan2025}, as well as by others in related context \cite{Flitton2004,Hosoi2004,Hewitt2015,Pedersen2019}, a linear sixth-order thin film equation can be derived in the long-wave (lubrication) limit \cite{Oron1997,Craster2009}, in which the equilibrium film is \emph{thin}, i.e., $h_0\ll\ell$, and remains thin for all time, i.e., $\max_{x,t} h(x,t)\ll \ell$. The governing equation can be written in the form of a conservation law:
\begin{subequations}\begin{align}
	\frac{\partial h}{\partial t} &= \frac{\partial q}{\partial x} , \label{eq:com}\\
	q &= \frac{h_0^3}{12\mu_f}\left( \rho_f g  \frac{\partial h}{\partial x} + B \frac{\partial^5 h}{\partial x^5} \right) . \label{eq:q}
\end{align}\label{eq:elastic_film}\end{subequations}
Here, \eqref{eq:com} expresses the conservation of mass (the so-called \emph{continuity equation}), and $q$ in \eqref{eq:q} is a \emph{flux} (area per time in this planar problem). 

It is convenient to make the governing equations~\eqref{eq:elastic_film} dimensionless using the following variable transformations:
\begin{equation}
	x = \ell\hat{x}, \qquad t = \frac{12\mu_f \ell^6}{Bh_0^3} \hat{t}, \qquad h = h_0 \hat{h}, \qquad q = \frac{Bh_0^4}{12\mu_f \ell^5} \hat{q} .
	\label{eq:elastic_film_ndvar}
\end{equation}
We have taken the dominant force on the interface to be due to bending resistance. Thus, the choice of the timescale for the problem is ${12\mu_f \ell^6}/(Bh_0^3)$, as in \cite{Tulchinsky2016,MartinezCalvo2020,NectarIvan2023,NectarIvan2025}.
Then, substituting the variables from \eqref{eq:elastic_film_ndvar} into \eqref{eq:elastic_film} and dropping the hats, we arrive at:
\begin{subequations}\begin{align}
	\frac{\partial h}{\partial t} &= \frac{\partial q}{\partial x} , \\
	q &= \mathcal{B} \frac{\partial h}{\partial x}  + \frac{\partial^5 h}{\partial x^5} . 
\end{align}\label{eq:elastic_film_nd3}\end{subequations}
Here, the elastic Bond number $\mathcal{B} = {\rho_f g \ell^4}/{B}$ quantifies the competition between gravity and bending resistance on the elastic interface (see also \cite{Duprat2011}).

Finally, we rewrite \eqref{eq:elastic_film_nd3} in terms of the dimensionless displacement $u$ from the equilibrium (i.e., letting $h=1+u$, where $u\ll1$) and combine the two subequations to arrive at the target sixth-order PDE featured in this work:
\begin{equation}
	\frac{\partial u}{\partial t} = \mathcal{B} \frac{\partial^2 u}{\partial x^2} + \frac{\partial^6 u}{\partial x^6}.
	\label{eq:elastic_film_linear}
\end{equation}

\subsection{Boundary conditions}

To solve \eqref{eq:elastic_film_linear}, we need six boundary conditions (BCs). Relevant boundary conditions for open troughs (domains) were discussed in \cite{Hosoi2004} in the context of a thin film with elastic resistance, and these were amplified and extended in \cite{NectarIvan2023} in the present context. Relevant boundary conditions for clamped films in closed troughs (domains) were discussed in \cite{Gabay2023}. Based on the latter results, we arrive at the six BCs (four on the film displacement and two on the flux) for our target problem of a clamped elastic film over a closed trough:
\begin{subequations}\label{eq:elastic_film_clamped_BC}
\begin{align}
	u|_{x=\pm1} &= 0, \label{eq:elastic_film_BC1} \\
	\left. \frac{\partial u}{\partial x}\right|_{x=\pm1} &= 0, \label{eq:elastic_film_BC2} \\
	q|_{x=\pm1} \equiv \left.\left(\mathcal{B} \frac{\partial u}{\partial x}  + \frac{\partial^5 u}{\partial x^5}\right)\right|_{x=\pm1} = \left. \frac{\partial^5 u}{\partial x^5}\right|_{x=\pm1} &= 0.
	\label{eq:elastic_film_BC3}
\end{align}\end{subequations}
Note that the flux conditions~\eqref{eq:elastic_film_BC3} are simplified by taking the clamping conditions~\eqref{eq:elastic_film_BC2} into account.

In the context of the PDE~\eqref{eq:elastic_film_linear} subject to the BCs~\eqref{eq:elastic_film_clamped_BC}, we consider the following sixth-order IBVP for $u=u(x,t)$:
\begin{subequations}
	\begin{empheq} [left = \empheqlbrace]{alignat=2}
		\frac{\partial u}{\partial t} - \mathcal{B} \frac{\partial^2 u}{\partial x^2} - \frac{\partial^6 u}{\partial x^6} &= f(x,t), \qquad &(x,t)\in(-1,1)\times\in(0,\infty), \label{eq:PDE_6}\\
		\left. u \right|_{x=\pm1} = \left.\frac{\partial u}{\partial x}\right|_{x=\pm1} = \left.\frac{\partial^5 u}{\partial x^5}\right|_{x=\pm1} &= 0,\quad &t\in(0,\infty), \label{eq:BC_6}\\
		u(x,0) &= u^0(x), \quad &x\in(-1,1),
	\end{empheq}\label{eq:IBVP_6}%
\end{subequations}
where $u^0(x)$ and $f(x,t)$ are given smooth functions.

To solve IBVP~\eqref{eq:IBVP_6}, we will employ a Petrov--Galerkin spectral method (section~\ref{sec:Galerkin}). To this end, we must construct suitable sets of adjoint eigenfunctions and determine their properties, which is the subject of the next section.

\section{The adjoint eigenvalue problems}\label{sec:EVP}

The eigenvalue problem (EVP) associated with IBVP~\eqref{eq:IBVP_6} reads:
\begin{subequations}
	\label{eq:6th_EVP}
	\begin{align}
		\label{eq:ODE_hg_6} 
		 -\frac{\rd^6 \psi}{\rd x^6} &= \lambda^6 \psi , 	\\
		\left.\psi \right|_{x=\pm1} = \left.\frac{\rd \psi}{\rd x}\right|_{x=\pm1} = \left.\frac{\rd^5 \psi}{\rd x^5}\right|_{x=\pm1} &=0 .
		\label{eq:6th_BCs}
	\end{align}
\end{subequations}	
Before we proceed with the derivation of the solutions (eigenfunctions) of  EVP~\eqref{eq:6th_EVP}, we first discuss why they possess certain properties required for the development of our Petrov--Galerkin spectral method in section~\ref{sec:Galerkin}.   

\subsection{Adjointness, biorthogonality, and expansions}\label{sub:Deriv_Biorth_Expand}

The required results are obtained by directly applying the general theory on $n$-th-order non-self-adjoint eigenvalue problems found in the comprehensive work of Birkhoff~\cite{Birkhoff1908}, but also in the classical monograph of Coddington and Levinson~\cite{CodLev}. We therefore facilitate this discussion for the reader by quoting or paraphrasing the relevant definitions and theorems from~\cite{Birkhoff1908} or~\cite{CodLev} accordingly. We adopt the same notation as~\cite{Birkhoff1908}. 

\begin{definition}{[Adjoint of a linear differential operator of order $n$ (from \cite{Birkhoff1908})]}
	 Consider the linear differential operator 
	 \begin{equation}
	 	\mathscr{L}_n(\,\cdot\,) := \frac{\rd^n}{\rd x^n}(\,\cdot\,) + p_1(x)\frac{\rd^{n-1}}{\rd x^{n-1}}(\,\cdot\,) + \cdots + (p_n(x)\,\cdot\,) ,
	 	\label{eq:n-th_order_operator}
	 \end{equation}	
	 where $p_1, \ldots, p_n$ are  real-valued functions of $x$ with continuous derivatives of all orders on a closed interval $[a,b]\subset\mathbb{R}$. 
	 Then, the adjoint of $\mathscr{L}_n $ is another linear operator, denoted $\mathscr{M}_n$, also of order $n$ given by
	 \begin{equation}
			\mathscr{M}_n(\,\cdot\,) := (-1)^n\frac{\rd^n}{\rd x^n}(\,\cdot\,) + (-1)^{n-1}\frac{\rd^{n-1}}{\rd x^{n-1}}(p_1(x)\,\cdot\,) + \cdots + (p_n(x)\,\cdot\,) .
		\label{eq:adjoint_n-th_order_operator}	
      \end{equation}
\end{definition}	

Throughout this work we will use the notation $\langle \cdot,\cdot \rangle$ for the  $L^2$ inner product on the interval $[a,b]$, that is, for any $\phi, \psi\in L^2[a,b]$, 
	\begin{equation}
		\langle \phi, \psi \rangle := \int_{a}^{b} \phi(x) \psi(x) \,\dif x .
	 \end{equation}
	 Also, $\|\cdot\|$ will denote the associated $L^2$ norm unless explicitly specified otherwise. 
	
\begin{definition}{[Adjoint of an $n$-th order eigenvalue problem  (from \cite{Birkhoff1908})]}
Let $\mathscr{L}_n$ be the operator given by \eqref{eq:n-th_order_operator}. The linear ordinary differential equation (ODE)
\begin{subequations}
	\label{eq:nth_EVP}
	\begin{equation}
		\label{eq:nth_ODE}
		\mathscr{L}_n(\psi) + \lambda \psi = 0 ,
	\end{equation}	
	subject to the linear homogeneous conditions in $\psi(a)$, $\psi'(a)$, $\ldots$, $\psi^{n-1}(a)$, $\psi(b)$, $\psi'(b)$, $\ldots$, $\psi^{n-1}(b)$,
	 \begin{equation}
	 	\label{eq:nth_BCs}
	 	\mathscr{U}_p(\psi) = 0 ,  \qquad (p=1,2, \ldots,n)
	 \end{equation}	
	forms an \emph{eigenvalue value problem} (EVP) for $\psi(x)$. 
\end{subequations}

The \emph{adjoint} EVP for $\phi(x)$ associated with EVP~\eqref{eq:nth_EVP} is formed by the ODE  
\begin{subequations}
	\label{eq:adjoint_nth_EVP}
	\begin{equation}
		\label{eq:nth_ODE_adjoint}
		\mathscr{M}_n(\phi) + \lambda \phi = 0 ,
	\end{equation}	
where $\mathscr{M}$ is given by \eqref{eq:adjoint_n-th_order_operator} and supplemented by $n$ adjoint boundary conditions in $\phi(a)$, $\phi'(a)$, $\ldots$, $\phi^{n-1}(a)$, $\phi(b)$, $\phi'(b)$, $\ldots$, $\phi^{n-1}(b)$,
	\begin{equation}
		\label{eq:nth_BCs_adjoint}
		\mathscr{V}_p(\phi) = 0 , \qquad (p=1,2, \ldots,n) .
	\end{equation}	
\end{subequations}
\end{definition}		

We do not discuss the form of the boundary conditions \eqref{eq:nth_BCs_adjoint} for the general $n$-th-order adjoint problem from \cite{Birkhoff1908}. Instead, later on in this section, we provide the (adjoint) BCs that are required to form the adjoint of the sixth-order EVP~\eqref{eq:6th_EVP}.  

We proceed with the following propositions (for proofs, see \cite{Birkhoff1908, CodLev}):
	
\begin{proposition}{[Common eigenvalues (from \cite{Birkhoff1908})]}
	\label{thm:prop_n-th_order_eigen}
	If a solution $\psi(x)$ of \eqref{eq:nth_EVP} exists for some eigenvalue $\lambda$, then a solution $\phi(x)$ of \eqref{eq:adjoint_nth_EVP} exists for the \emph{same eigenvalue $\lambda$}. Moreover, if $\psi(x)$ is unique (up to a constant multiplicative factor), then $\phi(x)$ is also unique (up to a constant multiplicative factor).
\end{proposition}	

\begin{proposition}{[Biorthogonality of the eigenfunctions of adjoint problems (from  \cite{Birkhoff1908, CodLev})]}
	\label{thm:prop_n-th_order_biorthog}
	  Consider the (common) eigenvalues 
	  \begin{equation*}
	  	    \lambda_1, \lambda_2, \ldots,
	  \end{equation*}
	  of EVPs~\eqref{eq:nth_EVP} and \eqref{eq:adjoint_nth_EVP}. Then their corresponding eigenfunctions $\big\{\psi_i\big\}_{i=1}^{\infty}$ of \eqref{eq:nth_EVP} and $\big\{\phi_i\big\}_{i=1}^{\infty}$ of \eqref{eq:adjoint_nth_EVP}
	  are \emph{biorthogonal}, namely
	   \begin{equation*}
	  	   	\langle \psi_i, \phi_j \rangle = \begin{cases}
	  		  0 & \text{if}\quad  i\neq j  \\
	  		c_{i}\neq 0 & \text{if}\quad  i= j
	  	\end{cases} 
	  	\,,
	  \end{equation*}
	  where the $c_i$ are constants.
\end{proposition}	

Finally, we can now state the expansion theorem (without proof): 	
\begin{proposition}{[Expansion theorem (paraphrasing \cite{Birkhoff1908} and \cite{CodLev})]}
	\label{thm:Exp_n-th_order}
	Let $f(x)$ be a continuous real-valued function on $[a,b]$, with a continuous first-order derivative and $\big\{\psi_i\big\}_{i=1}^{\infty}$, $\big\{\phi_i\big\}_{i=1}^{\infty}$ be as in proposition~\ref{thm:prop_n-th_order_biorthog}. Then, the series  
	\begin{equation}
		\label{eq:gen_expand_Birkhoff}
		\sum_{i=1}^{\infty} \frac{\D \langle f, \phi_i\rangle}{\D \langle \psi_i, \phi_i\rangle}\cdot \psi_i(x) ,
	\end{equation}	
	converges to $f(x)$, with the possible exception of $x=a$ and $x=b$.  
\end{proposition}

\subsection{The adjoint sixth-order problem and its boundary conditions}\label{}

Now, we return our focus to the sixth-order EVP~\eqref{eq:6th_EVP} and proceed with finding its adjoint problem. 
\begin{proposition}	
	\label{thm:6th_EVP_adj}
	The EVP 
	\begin{subequations}
		\label{eq:6th_EVP_adj}
		\begin{align}
			\label{eq:ODE_hg_6_adj} 
			-\frac{\rd^6 \phi}{\rd x^6} &= \lambda^6 \phi , 	\\
			\left.\frac{\rd \phi}{\rd x}\right|_{x=\pm1} = \left.\frac{\rd^2 \phi}{\rd x^2}\right|_{x=\pm1} = \left.\frac{\rd^3 \phi}{\rd x^3}\right|_{x=\pm1} &=0 ,
			\label{eq:6th_BCs_adj}
		\end{align}
		is the adjoint problem associated with EVP~\eqref{eq:6th_EVP}.	
	\end{subequations}	
\end{proposition}	

\begin{proof}
	Setting $n=6$ in equations~\eqref{eq:n-th_order_operator} and \eqref{eq:adjoint_n-th_order_operator} leads to the observation that the differential operator $\mathscr{L}= -{\rd^6}/{\rd x^6}$ is symmetric, that is, $\mathscr{M} = \mathscr{L}$. Therefore, to find the adjoint problem associated with EVP~\eqref{eq:6th_EVP}, we need only find the adjoint boundary conditions.
	
	Let $\psi(x)$ be a solution of EVP~\eqref{eq:6th_EVP} corresponding to some eigenvalue $\lambda$ and let $\phi(x)\in\mathcal{C}^6[-1,1]$ be the solution of ODE~\eqref{eq:ODE_hg_6_adj}  corresponding to the \emph{same eigenvalue $\lambda$} and satisfying some homogeneous boundary conditions of the form $\mathscr{V}_p(\phi)=0, \ \ p=1,\ldots,6$, which are to be determined.  The $L^2[-1,1]$ inner product between equation~\eqref{eq:ODE_hg_6} and $\phi$ is:
	\begin{equation}
		\int_{-1}^{+1} -\frac{\rd^6 \psi}{\rd x^6} \phi \,\rd x = \lambda^6 \int_{-1}^{+1} \psi\phi \,\rd x.
	\end{equation}
By repeated integration by parts, we obtain
\begin{subequations}
	\begin{equation}
		\langle \mathscr{L}\psi,\phi \rangle = \langle \psi,\mathscr{M}\phi \rangle + P(\psi,\phi)\Big|_{x=-1}^{x=+1}\,,
		\label{eq:Lpp_6}
	\end{equation}
	where, following Birkhoff \cite{Birkhoff1908}, $P(\psi,\phi)\Big|_{x=-1}^{x=+1}$ is a bilinear form of the values of $\phi$ and $\psi$ (and their derivatives) at $x=\pm1$, specifically:
	\begin{equation}
		P(\psi,\phi)\Big|_{x=-1}^{x=+1} := \left.\left[
		\frac{\rd^5\psi}{\rd x^5}\phi
		+ \frac{\rd \psi}{\rd x}\frac{\rd^4\phi}{\rd x^4}
		- \psi\frac{\rd^5\phi}{\rd x^5}
		- \frac{\rd^4 \psi}{\rd x^4}\frac{\rd\phi}{\rd x}
		+ \frac{\rd^3 \psi}{\rd x^3}\frac{\rd^2\phi}{\rd x^2}
		- \frac{\rd^2 \psi}{\rd x^2}\frac{\rd^3\phi}{\rd x^3}
		\right]\right|_{x=-1}^{x=+1} .
		\label{eq:Jpp_6}
	\end{equation}
\end{subequations}
	According to Birkhoff \cite{Birkhoff1908}, we can also write
	\begin{equation}
		P(\psi,\phi)\Big|_{x=-1}^{x=+1} = \sum_{p=1}^{12}\mathscr{U}_p(\psi)\mathscr{V}_{2n+1-p}(\phi),
		\label{eq:Jpp_6_2}
	\end{equation}
	where $\mathscr{U}$ and $\mathscr{V}$ were defined in \eqref{eq:nth_BCs} and \eqref{eq:nth_BCs_adjoint}, respectively. 
	The BCs~\eqref{eq:6th_BCs} inform us that $\mathscr{U}_{1,2}(\psi) = \psi^{(5)}(\pm1) = 0$, $\mathscr{U}_{3,4}(\psi) = \psi'(\pm1) = 0$, $\mathscr{U}_{5,6}(\psi) = \psi(\pm1) = 0$, whence for \emph{any} choice of $\mathscr{V}_{12,11}(\psi)$, $\mathscr{V}_{10,9}(\psi)$, $\mathscr{V}_{8,7}(\psi)$, \eqref{eq:Jpp_6} reduces to	
	\begin{equation}
		P(\psi,\phi)\Big|_{x=-1}^{x=+1} = \left.\left[
		- \frac{\rd^4 \psi}{\rd x^4}\frac{\rd\phi}{\rd x}
		+ \frac{\rd^3 \psi}{\rd x^3}\frac{\rd^2\phi}{\rd x^2}
		- \frac{\rd^2 \psi}{\rd x^2}\frac{\rd^3\phi}{\rd x^3}
		\right]\right|_{x=-1}^{x=+1} .
		\label{eq:Jpp_6_simplified}
	\end{equation}
	
	Choosing $\mathscr{V}_{6,5}(\phi) = \phi'(\pm1) = 0$, $\mathscr{V}_{4,3}(\phi) = \phi''(\pm1) = 0$, $\mathscr{V}_{2,1}(\phi) = \phi{'''}(\pm1) = 0$, makes $P(\psi,\phi)\Big|_{x=-1}^{x=+1}$ vanish for \emph{any} choice of $\mathscr{U}_{7,8}(\phi)$, $\mathscr{U}_{9,10}(\phi)$, $\mathscr{U}_{11,12}(\phi)$. These conditions are precisely the BCs~\eqref{eq:6th_BCs_adj} on $\phi$ we posited at the start.	Hence,  the adjointness condition
	\begin{equation}
		\langle \mathscr{L}\psi,\phi \rangle = \langle \psi,\mathscr{M}\phi \rangle,
		\label{eq:adj_cond_6}
	\end{equation}
	holds for the solutions pairs $\psi(x)$ of EVP~\eqref{eq:6th_EVP} and $\phi(x)$ of EVP~\eqref{thm:6th_EVP_adj} corresponding to any given common eigenvalue $\lambda$ thereof.  
	\end{proof}

\subsection{Constructing the eigenfunctions}

First, we solve EVP~\eqref{eq:6th_EVP} and derive our countable set of solutions. The characteristic equation corresponding to \eqref{eq:ODE_hg_6} has roots
\begin{equation}
	\left\{ +\lambda \ri, -\lambda \ri, \left(\tfrac{\ri}{2} + \tfrac{\sqrt{3}}{2}\right)\lambda, \left(\tfrac{\ri}{2} - \tfrac{\sqrt{3}}{2}\right)\lambda, \left(-\tfrac{\ri}{2} + \tfrac{\sqrt{3}}{2}\right)\lambda, \left(-\tfrac{\ri}{2} - \tfrac{\sqrt{3}}{2}\right)\lambda \right\}.
\end{equation} 
Therefore, the general solution of the homogeneous version of \eqref{eq:ODE_hg_6} can be written in the form
\begin{multline}
	\label{eq:gensol_ODE_hg_6} 
	\psi(x) = \overbrace{c_1 \sin(\lambda x) 
		+ c_5 \sin\left(\tfrac{1}{2} \lambda x\right) \cosh \left(\tfrac{\sqrt{3}}{2} \lambda x\right) 
		+ c_6 \cos\left(\tfrac{1}{2} \lambda x\right) \sinh \left(\tfrac{\sqrt{3}}{2} \lambda x\right)}^{\psi^s(x)\;\;\text{(odd)}}\\
	+ \underbrace{c_2 \cos(\lambda x)
		+ c_3 \sin\left(\tfrac{1}{2} \lambda x\right) \sinh \left(\tfrac{\sqrt{3}}{2} \lambda x\right)
		+ c_4 \cos\left(\tfrac{1}{2} \lambda x\right) \cosh \left(\tfrac{\sqrt{3}}{2} \lambda x\right)}_{\psi^c(x)\;\;\text{(even)}} .
\end{multline}
We choose to express the solution in this way, following the example of the so-called ``beam'' functions \cite{Chandrasek,Chri_Annuary_Chandra,PCB_IJNMF}, as well as the notation followed in~\cite{NectarIvan2023,NectarIvan2024,NectarIvan2025} in order to have odd and even sets of eigenfunctions resembling trigonometric sines and cosines.  

Imposing BCs~\eqref{eq:6th_BCs} on each of the $\psi^s(x)$ and $\psi^c(x)$ separately yields homogeneous linear systems for $c_1,c_5,c_6$ and $c_2,c_3,c_4$, respectively. To find nontrivial solutions, we set the determinants of the coefficient matrices of these linear systems equal to zero and arrive at the eigenvalue relations:
\begin{subequations}\label{eq:evals_sixth_015}\begin{alignat}{3}
		&\text{even}:\qquad & \cos{(2 \lambda^c)} + 2\cos{ \lambda^c} \cosh{(\sqrt{3} \lambda^c)} - 3  &= 0,  \label{eq:evals_even_0_1_5}\\
		&\text{odd}:\qquad & \sin{\lambda^s} \left(\cos{\lambda^s} - \cosh{(\sqrt{3} \lambda^s)}\right)  &= 0.  \label{eq:evals_odd_0_1_5}
\end{alignat}\end{subequations}

Each solution $\lambda_m^c$ of  \eqref{eq:evals_even_0_1_5} is an eigenvalue corresponding to an even eigenfunction $\psi^c_m$ of EVP~\eqref{eq:6th_EVP}, whereas each solution $\lambda_m^s$ of  \eqref{eq:evals_odd_0_1_5} corresponds to an odd eigenfunction $\psi^s_m$. The solutions of the eigenvalue relation \eqref{eq:evals_odd_0_1_5} are the positive roots of $\sin{\lambda_m^s}=0$, therefore the ``odd'' eigenvalues are given by $\lambda_m^s = m\pi$ with $m\in\mathbb{N}$.

\begin{table}[h!]
	\caption{\label{tab:eigenvalues} The ``even''  eigenvalues $\lambda_m^c$, found by solving the transcendental equation~\eqref{eq:evals_even_0_1_5}.  The numerical solutions using \textsc{Mathematica}'s   \texttt{FindRoot} \cite{Mathematica} are listed alongside the corresponding values from the proposed asymptotic formula.}
	\centering 
	\begin{tabular}{l*{4}{d{12}}}    
		\toprule
		\multicolumn{1}{c}{$m$} & \multicolumn{1}{c}{$\lambda_m^c$} & \multicolumn{1}{c}{$(m+1/2)\pi$} \\
		\midrule
		\multicolumn{1}{l|}{0} & 0 & - \\
		\multicolumn{1}{l|}{1} & 4.71352778544 & 4.71238898038 \\
		\multicolumn{1}{l|}{2} & 7.85397668926 & 7.85398163397 \\
		\multicolumn{1}{l|}{3} & 10.9955743090 & 10.9955742876 \\
		\multicolumn{1}{l|}{4} & 14.1371669411 & 14.1371669412 \\
		\multicolumn{1}{l|}{5} & 17.2787595947 & 17.2787595947 \\
		\multicolumn{1}{l|}{6} & 20.4203522483 & 20.4203522483 \\
		\bottomrule
	\end{tabular}
\end{table}

On the other hand, eigenvalue relation \eqref{eq:evals_even_0_1_5} was solved numerically using a highly accurate numerical solver--\textsc{Mathematica}'s \texttt{FindRoot} \cite{Mathematica} with 16 digits of working precision. Utilizing~\eqref{eq:evals_even_0_1_5} we also derived large-$\lambda^c$ asymptotic formulas for the ``even'' eigenvalues, finding that  $\lambda_m^c \sim (m+1/2)\pi$ as $\lambda^c\to\infty$. The results are presented in table~\ref{tab:eigenvalues}. As we can see, the asymptotic formula is extremely accurate even for  $m=5$, with the formula for $\lambda_m^c$ agreeing with the results obtained with \texttt{FindRoot} for 12 digits. Thus, in practice, when implementing a Petrov--Galerkin spectral method, there will be no need to use numerical root-finding to determine $\lambda_m^c$ for $m\geq6$. At this point, we note that our calculations and derived asymptotic formulas demonstrate the fact that \emph{all the eigenvalues $\lambda_m^c,\ m\geq 0$ and $\lambda_m^s,\ m\geq 1$ of~\eqref{eq:6th_EVP} are distinct}.

What remains is to determine the constants $c_2$, $c_3$, and $c_4$ for the even eigenfunctions $\psi_m^c(x)$ and $c_1$, $c_5$ and $c_6$ for the odd eigenfunctions $\psi_m^s(x)$, which were introduced in \eqref{eq:gensol_ODE_hg_6}. These constants are found by imposing BCs~\eqref{eq:6th_BCs} and utilizing the eigenvalue relations~\eqref{eq:evals_sixth_015}. We do not normalize the eigenfunctions but instead select the multiplicative constant to be one. After a lengthy calculation, and with the aid of various \textsc{Mathematica}~\cite{Mathematica} algebraic and trigonometric manipulation routines, we arrive at the expressions: 
\begin{subequations}
	\label{eq:efuncs_6_015}
	\begin{align}
		\psi_m^c(x) &= 
	     \cos{\lambda_m^c x} \nonumber \\ 
	     &\phantom{=}+ \tfrac{\cos{\tfrac{\lambda_m^c}{2}} \cosh{\tfrac{\sqrt{3}\,\lambda_m^c}{2}}}{\sqrt{3}\,\sin{\lambda_m^c} - \sinh{(\sqrt{3}\,\lambda_m^c)}}  \Bigg[\bigg(-2\sqrt{3}\cos{\lambda_m^c} \tan{\tfrac{\lambda_m^c}{2}} - 2 (\cos{\lambda_m^c}-2) \tanh{\tfrac{\sqrt{3}\,\lambda_m^c}{2}}\bigg) \cos{\tfrac{\lambda_m^c}{2}  x} \cosh{\tfrac{\sqrt{3}\,\lambda_m^c}{2} x} \nonumber\\ 
		 &\phantom{=}+\bigg( -3 \tan{\tfrac{\lambda_m^c}{2}} + \sin^2{\tfrac{\lambda_m^c}{2}} \tan{\tfrac{\lambda_m^c}{2}} - \frac{3}{2}\sin {\lambda_m^c} + 2 \sqrt{3}\cos{\lambda_m^c} \tanh{\tfrac{\sqrt{3} \lambda_m^c}{2}}\bigg)\sin{\tfrac{\lambda_m^c}{2} x} \sinh{\tfrac{\sqrt{3}\, \lambda_m^c }{2}  x}\Bigg]   	 \,,\label{eq:efuncs_6_015_even} \\ 
		\psi_m^s(x) &= \sin{\lambda_m^s x} - \tfrac{4}{\sinh{(\sqrt{3}\,\lambda_m^s)}}  \Bigg[\bigg(\cos^3{\tfrac{\lambda_m^s}{2}}\sinh{\tfrac{\sqrt{3}\,\lambda_m^s}{2}}\bigg)  \sin{\tfrac{\lambda_m^s}{2} x} \cosh{\tfrac{\sqrt{3}\,\lambda_m^s}{2}  x} \nonumber\\ 
		&\phantom{=}+\bigg( \sin^3{\tfrac{\lambda_m^s}{2}}\cosh{\tfrac{\sqrt{3}\,\lambda_m^s}{2}}\bigg)\cos{\tfrac{\lambda_m^s}{2} x} \sinh{\tfrac{\sqrt{3}\, \lambda_m^s}{2}  x}\Bigg]\,, \label{eq:efuncs_6_015_odd} 
	\end{align}
	 for $m\in\mathbb{N}$\,.
\end{subequations}

\begin{figure}[ht]
	\centering
	\begin{subfigure}[b]{0.475\textwidth}
		\centering
		\includegraphics[width=\textwidth]{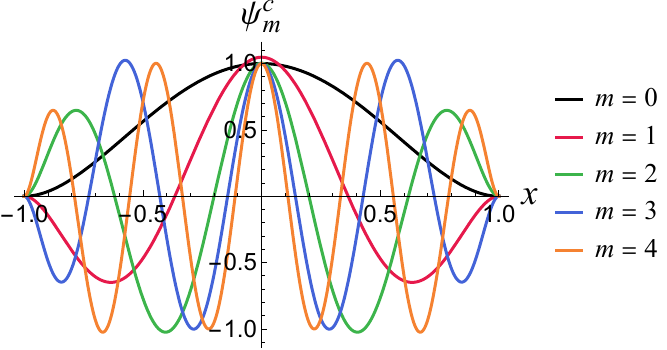}
		\caption{Even eigenfunctions.~\label{fig:efunc_015_even}}
	\end{subfigure}
	\hfill
	\begin{subfigure}[b]{0.475\textwidth}
		\centering
		\includegraphics[width=\textwidth]{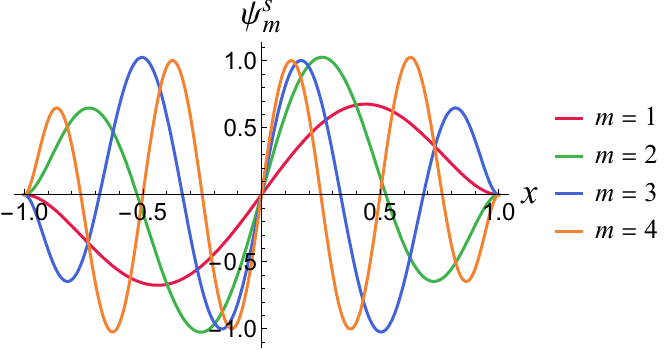}
		\caption{Odd eigenfunctions.~\label{fig:efunc_015_odd}}
	\end{subfigure}
	\caption{The profiles of the (a) even and (b) odd eigenfunctions of EVP~\eqref{eq:6th_EVP} for $m=1,2,3,4$. The even eigenfunction $\psi_0^c(x)=x^4-2 x^2+1$, which corresponds to $\lambda_0^c=0$, for $m=0$ is shown in (a) along with the even ones.~\label{fig:efunc_015}}
\end{figure}
It is important to note here that $\lambda_0^c=0$ is an eigenvalue of EVP~\eqref{eq:6th_EVP} with corresponding eigenfunction 
\begin{equation}
	\psi_0^c(x)=x^4-2 x^2+1.
\end{equation}
The notation is due to the fact that $x^4-2 x^2+1$ is an even function of $x$.  The profiles of the first few members of the two sets of eigenfunctions are presented in figure~\ref{fig:efunc_015}.

We now proceed with the solution of EVP~\eqref{eq:6th_EVP_adj}. To explicitly demonstrate proposition~\ref{thm:prop_n-th_order_eigen}, we do not assume that the eigenvalues of   problem~\eqref{eq:6th_EVP_adj} are the same as those of EVP~\eqref{eq:6th_EVP}, and so we denote them by $\mu$. Therefore, ODE~\eqref{eq:ODE_hg_6_adj} now reads 
\begin{equation}
	    \label{eq:ODE_hg_6_adj_mu}
		-\frac{\rd^6 \phi}{\rd x^6} = \mu^6 \phi .
\end{equation}	
Following the same process as in the derivation of \eqref{eq:gensol_ODE_hg_6}, we write the general solution of the  homogeneous ODE~\eqref{eq:ODE_hg_6_adj_mu}  in the form
\begin{multline}
	\label{eq:gensol_ODE_hg_6_adj} 
	\phi(x) = \overbrace{c_1 \sin(\mu x) 
		+ c_5 \sin\left(\tfrac{1}{2} \mu x\right) \cosh \left(\tfrac{\sqrt{3}}{2} \mu x\right) 
		+ c_6 \cos\left(\tfrac{1}{2}\mu x\right) \sinh \left(\tfrac{\sqrt{3}}{2} \mu x\right)}^{\phi^s(x)\;\;\text{(odd)}}\\
	+ \underbrace{c_2 \cos(\lambda x)
		+ c_3 \sin\left(\tfrac{1}{2} \mu x\right) \sinh \left(\tfrac{\sqrt{3}}{2} \mu x\right)
		+ c_4 \cos\left(\tfrac{1}{2} \mu x\right) \cosh \left(\tfrac{\sqrt{3}}{2} \mu x\right)}_{\phi^c(x)\;\;\text{(even)}} \,,
\end{multline}
where $\phi^s(x)$ and $\phi^c(x)$ denote odd and even functions of $x$, respectively. We note that the constants $c_1,\ldots,c_6$ in \eqref{eq:gensol_ODE_hg_6_adj} are \emph{different} from those that appear in \eqref{eq:gensol_ODE_hg_6}. To evaluate these constants, we impose BCs~\eqref{eq:6th_BCs_adj} on $\phi^s(x)$ and $\phi^c(x)$ separately. This yields homogeneous linear systems for $c_1,c_5,c_6$ and $c_2,c_3,c_4$, respectively. To find nontrivial solutions, we set the determinants of the coefficient matrices of these linear systems equal to zero and arrive at the eigenvalue relations:
\begin{subequations}
\label{eq:evals_sixth_015_adj}
\begin{alignat}{3}
		&\text{even}:\qquad & \cos{(2 \mu^c)} + 2\cos{ \mu^c} \cosh{(\sqrt{3} \mu^c)} - 3  &= 0\,,  \label{eq:evals_even_1_2_3}\\
		&\text{odd}:\qquad & \sin{\mu^s} \left(\cos{\mu^s} - \cosh{(\sqrt{3} \mu^s)}\right)  &= 0\,,  \label{eq:evals_odd_1_2_3}
\end{alignat}
\end{subequations}
which are the same as \eqref{eq:evals_sixth_015}, confirming the prediction of the theory that adjoint EVPs have equal eigenvalues (see proposition~\ref{thm:prop_n-th_order_eigen}). Therefore, once again the ``odd'' eigenvalues are given by $\mu_m^s=m\pi (\equiv\lambda_m^s)$ for $m\in\mathbb{N}$, while the ``even'' eigenvalues are $\mu_m^c=\lambda_m^c$ for $m\in\mathbb{N}$ (see table~\ref{tab:eigenvalues} for values and the asymptotic relation). As expected, $\mu_0^c=0$ is an eigenvalue of  EVP~\eqref{eq:6th_EVP_adj}. Its corresponding eigenfunction is $\phi_0^c(x)=1$. 

To determine the constants $c_2$, $c_3$, and $c_4$ for the even eigenfunctions $\phi_m^c(x)$ and $c_1$, $c_5$, and $c_6$ for the odd eigenfunction $\phi_m^s(x)$, which appear in \eqref{eq:gensol_ODE_hg_6_adj}, we follow the process used for deriving \eqref{eq:efuncs_6_015}. The difference is that here we impose BCs~\eqref{eq:6th_BCs_adj} instead.  We find
\begin{subequations}
	\label{eq:efuncs_6_123}
	\begin{align}
		\phi_m^c(x) &= 
		\cos{\mu_m^c x} 
		+ \Bigg[\tfrac{2 \sin{\mu_m^c} \bigg(\cos{\tfrac{\mu_m^c}{2}} \sinh{ \tfrac{\sqrt{3}\, \mu_m^c}{2}} - \sqrt{3} \sin{\tfrac{\mu_m^c }{2}} \cosh{\tfrac{\sqrt{3}\, \mu_m^c}{2}}\bigg)}{\cos{\mu_m^c} - \cosh{\sqrt{3}\mu_m^c}} \Bigg] \sin{\tfrac{\mu_m^c}{2} x} \sinh{\tfrac{\sqrt{3}\, \mu_m^c }{2}  x} \nonumber\\ 
		&\phantom{=}+ \Bigg[\tfrac{\csc {\tfrac{\mu_m^c }{2}} \sech{\tfrac{\sqrt{3}\, \mu_m^c}{2}}\sin{\mu_m^c} \bigg(\sqrt{3} \cot{\tfrac{\mu_m^c }{2}}\tanh{\tfrac{\sqrt{3}\, \mu_m^c}{2}} + 1\bigg)}{1 + \cot ^2{\tfrac{\mu_m^c}{2}} \tanh ^2{\tfrac{\sqrt{3}\, \mu_m^c}{2}}} \Bigg] \cos{\tfrac{\mu_m^c}{2}  x} \cosh{\tfrac{\sqrt{3}\,\mu_m^c}{2} x}	 \,,\label{eq:efuncs_6_123_even} \\ 
		\phi_m^s(x) &= \sin{\mu_m^s x} + \Bigg[\tfrac{2 \cos{\mu_m^s} \bigg(\cos{\tfrac{\mu_m^s}{2}} \cosh{ \tfrac{\sqrt{3}\, \mu_m^s}{2}} - \sqrt{3} \sin{\tfrac{\mu_m^s}{2}} \sinh{\tfrac{\sqrt{3}\, \mu_m^s}{2}}\bigg)}{\cos{\mu_m^s} + \cosh{\sqrt{3}\,\mu_m^s}} \Bigg] \sin{\tfrac{\mu_m^s}{2} x} \cosh{\tfrac{\sqrt{3}\,\mu_m^s}{2}  x} \nonumber\\ 
		&\phantom{=}-	\Bigg[\tfrac{\cos{\mu_m^s} \csch{ \tfrac{\sqrt{3}\, \mu_m^s}{2}} \bigg(\sin{\tfrac{\mu_m^s}{2}} + \sqrt{3} \cos{\tfrac{\mu_m^s }{2}} \coth{\tfrac{\sqrt{3}\, \mu_m^s}{2}}\bigg)}{\sin^2{\tfrac{\mu_m^s}{2}} + \cos^2{\tfrac{\mu_m^s }{2}}\coth^2{\tfrac{\sqrt{3}\,\mu_m^s}{2}}} \Bigg]\cos{\tfrac{\mu_m^s}{2} x} \sinh{\tfrac{\sqrt{3}\, \mu_m^s }{2}  x} \,, \label{eq:efuncs_6_123_odd} 
	\end{align}
\end{subequations}
for $m\in\mathbb{N}$.

The first few members of each adjoint eigenfunction set are plotted in Fig.~\ref{fig:efunc_123}.
\begin{figure}[h!]
	\centering
	\begin{subfigure}[b]{0.45\textwidth}
		\centering
		\includegraphics[width=\textwidth]{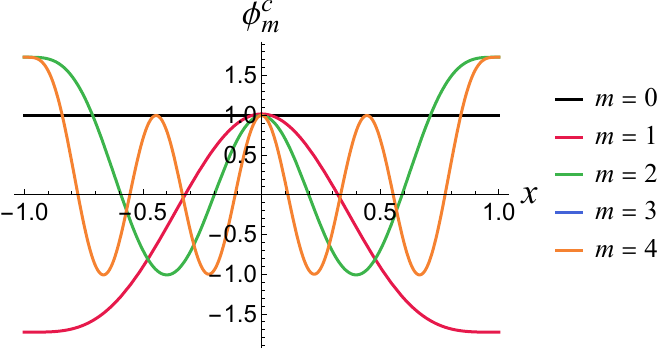}
		\caption{Even eigenfunctions.~\label{fig:efunc_123_even}}
	\end{subfigure}
	\hfill
	\begin{subfigure}[b]{0.45\textwidth}
		\centering
		\includegraphics[width=\textwidth]{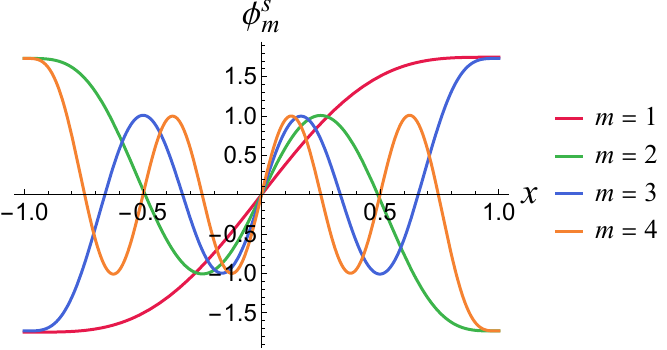}
		\caption{Odd eigenfunctions.~\label{fig:efunc_123_odd}}
	\end{subfigure}
	\caption{The profiles for $m=1,2,3,4$ of the (a) even, and (b) odd eigenfunctions of the adjoint EVP~\eqref{eq:6th_EVP_adj} associated with~\eqref{eq:6th_EVP}. The even eigenfunction $\phi_0^c(x)=1$, which corresponds to $\mu_0^c=0$, for $m=0$ is shown in (a) along with the even ones.~\label{fig:efunc_123}}
\end{figure}

Having derived our sets of eigenfunctions for the adjoint EVPs~\eqref{eq:6th_EVP} and \eqref{eq:6th_EVP_adj}, we apply proposition~\ref{thm:prop_n-th_order_biorthog} to obtain the following.
\begin{proposition}{[Biorthogonality of the eigenfunctions of the sixth-order adjoint EVPs]}
	\label{thm:prop_6-th_order_biorthog}
	\ Consider the sets of even $\Bigl\{\lambda_\ell^c\Bigr\}_{\ell=0}^{\infty}$ and odd $\Bigl\{\lambda_\ell^s\Bigr\}_{\ell=1}^{\infty}$ eigenvalues of EVP~\eqref{eq:6th_EVP},  and their counterparts $\Bigl\{\mu_\ell^c\Bigr\}_{\ell=0}^{\infty}$ and $\Bigl\{\mu_\ell^s\Bigr\}_{\ell=1}^{\infty}$ of its adjoint problem\!~ \eqref{eq:6th_EVP_adj}. Since $\lambda_\ell^c=\mu_\ell^c$ \ $\forall \ell\in\mathbb{N}\cup\{0\}$ and  $\lambda_\ell^s=\mu_\ell^s$ \ $\forall \ell \in \mathbb{N}$, then the following biorthogonality relations hold for their corresponding eigenfunctions, 
\begin{subequations}
	\label{eq:biorth_six}
	\begin{align}
		\langle \psi_\ell^c, \phi_m^c \rangle&= \begin{cases}
			0 & \text{if}\quad  \ell\neq m  \\
			c_{m}\neq 0 & \text{if}\quad  \ell = m
		\end{cases} 
		\,, \\[2mm]
		\langle \psi_\ell^s, \phi_m^s \rangle&= \begin{cases}
			0 & \text{if}\quad  \ell \neq m  \\
			s_{m}\neq 0 & \text{if}\quad  \ell = m
		\end{cases} 
		\,,
	\end{align}
\end{subequations}	
	where the $c_m$ and $s_m$ are given by the formulas
	\begin{subequations}
		\label{eq:biorth_const}
		\begin{align}
			c_0&= \left\langle \psi_0^c(x), \phi_0^c(x) \right\rangle =  \int_{-1}^1 (x^4 - 2x^2 +1) \, \rd x = 16/15 , \label{eq:biorthog_0} \\
			c_m&=\tfrac{\csc^2{\tfrac{\lambda_m^c}{2}}  \sech^2{\tfrac{\sqrt{3}\,\lambda_m^c}{2}}}{24\lambda_m^c  \left(\sqrt{3}\sin {\lambda_m^c} - \sinh{(\sqrt{3}\lambda_m^c)} \right) \left( 1 + \cot ^2{\tfrac{\lambda_m^c}{2}} \tanh^2{\tfrac{\sqrt{3}\,\lambda_m^c}{2}}\right)} \label{eq:biorthog_even}\\
			&\phantom{=} \times \Bigg[4 \sqrt{3} \cos ^2{\lambda_m^c} \cosh{(2\sqrt{3}\lambda_m^c)} - 2\sqrt{3}\cosh{(\!\sqrt{3}\,\lambda_m^c)} \big( -18\lambda_m^c  \sin{\lambda_m^c} + 7 \cos{\lambda_m^c} + \cos{3\lambda_m^c} \ \big) \nonumber \\
			&\qquad - 2\!\sqrt{3}\, \big(\, 9\lambda_m^c \sin{2\lambda_m^c}  - 7\cos{2\lambda_m^c} + \cos{4\lambda_m^c}\, \big)  +  6 \sinh{(\sqrt{3}\lambda_m^c)} \big(7 \sin{\lambda_m^c} - \sin{3\lambda_m^c} + 2 \lambda_m^c \cos{3 \lambda}   \  \big) \nonumber \\
			&\qquad - 6 \big(\lambda_m^c +\sin{2\lambda_m^c}\big) \sinh{(2\sqrt{3}\lambda_m^c)} \Bigg]\,,  \nonumber \\
			s_m&=-\tfrac{\csch{\tfrac{\sqrt{3}\,\lambda_m^c}{2}}  \sech{\tfrac{\sqrt{3}\,\lambda_m^c}{2}}}{12\lambda_m^c \big( \cos{\lambda_m^c} +  \cosh{(\sqrt{3}\lambda_m^c)} \big)}  \label{eq:biorthog_odd}\\
			&\phantom{=}\times \Bigg[ - 2\sqrt{3}\cos{2\lambda_m^c} \cosh^2{(\sqrt{3}\lambda_m^c)} +\sqrt{3}\sin{2\lambda_m^c} \cosh{(\sqrt{3}\lambda_m^c)} \big(2 \sin{\lambda_m^c} - 3 \lambda_m^c \cos{\lambda_m^c}\big)     \nonumber  \\
			&\phantom{=}+\sqrt{3} \big(3 \lambda_m^c \sin{2\lambda_m^c} + \cos{4\lambda_m^c} + \cosh{(2 \sqrt{3}\lambda_m^c)}\big)  + 6 \lambda_m^c \cos^3{\lambda_m^c}\sinh{(\sqrt{3}\lambda_m^c)} - 3\lambda_m^c \sinh{(2 \sqrt{3}\lambda_m^c)} \Bigg] . \nonumber
		\end{align}
	\end{subequations}
\end{proposition}	
Expressions~\eqref{eq:biorthog_even} and \eqref{eq:biorthog_odd} are cumbersome, however, for $m\geq7$ both $c_m=1$ and $s_m=1$ to \emph{at least 16 digits} after the decimal point (absolute error $\leq 10^{-16}$).  This simplifies the calculation of the spectral coefficients in \eqref{eq:full_expansion_formula} below, improving the efficiency of the proposed computational method. 
		
We conclude this section by stating the following proposition, which forms the foundation of our spectral method.
	\begin{proposition}{[Expansion into sixth-order biorthogonal sets of functions (applying \cite{CodLev})]}
		\label{prop:EVP_beam_6}
		The two sets of solutions $\psi_m^c(x)$ and $\psi_m^s(x)$ of EVP~\eqref{eq:6th_EVP} given by \eqref{eq:efuncs_6_015} and supplemented by $\psi_0^c(x) = (x)=x^4-2 x^2+1$, along with the corresponding solution sets $\phi_m^c(x)$ and $\phi_m^s(x)$ of its adjoint EVP~\eqref{eq:6th_EVP_adj} given by \eqref{eq:efuncs_6_123} and supplemented by $\phi_0^c(x) =  1$, form a \emph{complete biorthogonal set of functions} on $L^2[-1,1]$. Therefore, any function $u(x)\in L^2[-1,1]$ can be expanded as 
		\begin{subequations}
			\label{eq:full_expansion_formula} 
			\begin{alignat}{2}
				u(x) &=  u_0^c \psi_0^c(x) + \sum_{m=1}^\infty u_m^c \psi_m^c(x) + u_m^s \psi_m^s(x), \quad & \label{eq:u_expansion_formula}\\
				u_0^c &:= \frac{1}{2c_0}\big\langle u(x), \phi_0^c(x) \big\rangle = \frac{1}{2c_0}\int_{-1}^{+1} u(x) \, \rd x, \quad &\\
				u_m^c &:= \frac{1}{c_m}\left\langle u(x), \phi_m^c(x) \right\rangle = \frac{1}{c_m}\int_{-1}^{+1} u(x) \phi_m^c(x) \,\rd x,\quad &m\ge 1,\\
				u_m^s &:= \frac{1}{s_m}\left\langle u(x), \phi_m^s(x) \right\rangle = \frac{1}{s_m}\int_{-1}^{+1} u(x) \phi_m^s(x) \,\rd x,\quad &m\ge 1,
			\end{alignat}
		\end{subequations}
where $c_0$, $c_m$ and $s_m$ are given by \eqref{eq:biorth_const} and the series \eqref{eq:u_expansion_formula} converges to $u(x)$ in $L^2[-1,1]$.
	\end{proposition}	
This proposition follows from propositions~\ref{thm:prop_n-th_order_biorthog}, \ref{thm:Exp_n-th_order}, and Chapter 12 of \cite{CodLev}. Crucially, the relevant theorems in \cite{CodLev} were applied after taking into account the fact that (i) BCs~\eqref{eq:6th_BCs} and \eqref{eq:6th_BCs_adj} are homogeneous, (ii) the eigenvalues of \eqref{eq:6th_EVP} (and hence \eqref{eq:6th_EVP_adj}) are distinct, and (iii) the sets of eigenfunctions $\Bigl\{\psi_0^c\Bigr\} \cup \Bigl\{\psi_m^c,\, \psi_m^s\Bigr\}_{m=1}^{\infty}$ of~ \eqref{eq:6th_EVP} and $\Bigl\{\phi_0^c\Bigr\} \cup \Bigl\{\phi_m^c,\, \phi_m^s\Bigr\}_{m=1}^{\infty}$ of \eqref{eq:6th_EVP_adj} are linearly independent.

\section{The Petrov--Galerkin spectral method}\label{sec:Galerkin}

In the previous section, $n$ indicated the order of the EVP, and $m$ was the index used for the countable infinity of eigenvalues and eigenfunctions. We have now restricted ourselves to a specific sixth-order EVP, so in the following, we reuse $n$ as the index for the countable infinity of eigenvalues and eigenfunctions to simplify the notation.

\subsection{Method formulation}\label{subsec:MethForm}
Having established the necessary theoretical foundation and derived our biorthogonal sets of eigenfunctions, we proceed with the development of the proposed Petrov--Galerkin spectral method~\cite{Shen} for IBVP~\eqref{eq:IBVP_6}. To solve the sixth-order PDE~\eqref{eq:PDE_6} using a Petrov--Galerkin method based on the results above, we expand the sought function (i.e., the solution of the IBVP) $u(x,t)$ into our set of \emph{trial functions} $\psi_0^c(x)=1$, $\big\{\psi_{n}^c(x)\big\}_{n=1}^{M}$, and $\big\{\psi_{n}^s(x)\big\}_{n=1}^{M}$ as in \eqref{eq:full_expansion_formula}, allowing the expansion coefficients to be functions of time, and truncate the series at $M$ terms:
\begin{equation}
	u(x,t) \approx u_{\mathrm{spectral}}(x,t) = u_0^c(t) \psi_0^c(x) + \sum_{n=1}^{M} u_n^c(t) \psi_n^c(x) + u_n^s(t) \psi_n^s(x).	
	\label{eq:u_x_t_expansion_formula}
\end{equation}
Note that the spectral expansion~\eqref{eq:u_x_t_expansion_formula} intrinsically satisfies BCs~\eqref{eq:BC_6}, which is an important advantage of both classical Galerkin and Petrov--Galerkin approaches.

Next, since PDE~\eqref{eq:PDE_6} contains the term ${\D\partial^2 u}/{\D\partial x^2}$, the second derivatives of the basis functions must be expressed as linear combinations of the same basis functions:
\begin{subequations}\label{eq:beta_expansion_series}
	\begin{alignat}{2}
		\frac{\rd^2\psi_n^{c}}{\rd x^2}  &= \sum_{m=1}^\infty \frac{\D 1}{\D c_m}\,\beta_{nm}^{c} \psi_{m}^{c}(x) , \qquad\text{where}\quad &\beta_{nm}^{c} &:= \left\langle \frac{\rd^2\psi_n^{c}}{\rd x^2}  ,\phi_m^{c} \right\rangle = \int_{-1}^{+1} \frac{\rd^2\psi_n^{c}}{\rd x^2} \psi_m^{c} \,\rd x,  \label{eq:beta_series_even}\\
		\frac{\rd^2\psi_n^{s}}{\rd x^2}  &= \sum_{m=1}^\infty  \frac{\D 1}{\D s_m}\,\beta_{nm}^{s} \psi_{m}^{s}(x), \qquad\text{where}\quad &\beta_{nm}^{s} &:= \left\langle \frac{\rd^2\psi_n^{s}}{\rd x^2}  ,\phi_m^{s} \right\rangle = \int_{-1}^{+1} \frac{\rd^2\psi_n^{s}}{\rd x^2} \psi_m^{s} \,\rd x,  \label{eq:beta_series_odd}
	\end{alignat}
\end{subequations}	
and $c_m$,  $s_m$ are given in \eqref{eq:biorth_const}. Note that whilst $\beta_{n0}^{s}=0$ by definition,  we find  $\beta_{n0}^c=0$ by computing the integral $\int_{-1}^{+1} \frac{\rd^2\psi_n^{c}}{\rd x^2}\,\rd x$. The expressions for $\beta_{nm}^{c,s}$ will be given in section~\ref{sec:2nd_der_expansion_formulas} below. It follows that
\begin{subequations}\label{eq:deriv_spectral_series}
	\begin{align}
		\frac{\partial^2 u_{\mathrm{spectral}}}{\partial x^2}  &=  \sum_{n=1}^M \sum_{m=1}^M  u_n^c(t) \frac{\D 1}{\D c_m}\,\beta_{nm}^{c} \psi_{m}^{c}(x) + \sum_{n=1}^M \sum_{m=1}^M u_n^s(t) \frac{\D 1}{\D s_m}\,\beta_{nm}^{s} \psi_{m}^{s}(x),\\
		\frac{\partial^6 u_{\mathrm{spectral}}}{\partial x^6}  &=  -\sum_{n=1}^M (\lambda^{c}_n)^6 u_n^c(t) \psi_n^c(x) - \sum_{n=1}^M (\lambda^{s}_n)^6 u_n^s(t) \psi_{n}^{s}(x).
	\end{align}%
\end{subequations}

Next, substituting \eqref{eq:u_x_t_expansion_formula} and \eqref{eq:deriv_spectral_series} into \eqref{eq:PDE_6}, taking successive inner products with the \emph{test functions} $\phi_0^c(x)=1$, $\big\{\phi_{\ell}^c(x)\big\}_{\ell=1}^{M}$, and $\big\{\phi_{\ell}^s(x)\big\}_{\ell=1}^{M}$, and using the biorthogonality property~\eqref{eq:biorth_six} of the sets of eigenfunctions, we obtain the semi-discrete (dynamical) system:
\begin{subequations}\label{eq:dynamical_system}
	\begin{alignat}{2}
		c_0\frac{\rd u_0^c}{\rd t} &= f_0^c(t), \quad &(\ell=0), \label{eq:dyn_sys_0}\\
		c_{\ell}\frac{\rd u_{\ell}^c}{\rd t} &=  \sum_{n=1}^M \left[  \mathcal{B}\,\beta_{n \ell}^{c} - c_{\ell}(\lambda^{c}_n)^6 \delta_{n \ell} \right] u_n^c(t) + f_\ell^c(t), \qquad &\ell\ge 1, \label{eq:dyn_sys_c}\\
		s_{\ell}\frac{\rd u_{\ell}^s}{\rd t} &= \sum_{n=1}^M \left[ \mathcal{B}\,\beta_{n \ell}^{s} - s_{\ell} (\lambda^{s}_n)^6 \delta_{n \ell} \right] u_n^s(t) + f_\ell^s(t), \qquad &\ell\ge 1. \label{eq:dyn_sys_s}
	\end{alignat}
\end{subequations}
where $\Big\{f_0^c(t), f_\ell^c(t),  f_\ell^s(t) \Big\}$ are the inner products of the known function $f(x,t)$ on the right-hand side of \eqref{eq:PDE_6} with $\phi_0^c(x)=1$, $\Big\{\phi_{\ell}^c(x)\Big\}_{\ell=1}^{M}$ and $\Big\{\phi_{\ell}^s(x)\Big\}_{\ell=1}^{M}$ respectively. For the second derivative term and the right-hand side term, the constants $c_{\ell}$ and $s_{\ell}$ simplify; however, this is not the case for the remaining terms. Nonetheless, as discussed previously, we need only include them for $\ell<7$. 

System~\eqref{eq:dynamical_system} can be solved using a variety of techniques (e.g., following  \cite{NectarIvan2025}), but its solution is the subject of future work. We conclude the current discussion by noting that equation~\eqref{eq:dyn_sys_0} for the zeroth coefficient $u_0^c$, the system~\eqref{eq:dyn_sys_c} for the coefficients $u_\ell^c$ of the even eigenfunctions and the corresponding system~\eqref{eq:dyn_sys_s} for the coefficients $u_\ell^s$ of the odd eigenfunctions are decoupled; this means that we must solve two $M\times M$ systems plus one ODE rather than a $(2M+1)\times(2M+1)$ system. This observation makes the solution approach more efficient.

\subsection{Second-derivative expansion formulas}\label{sec:2nd_der_expansion_formulas}

We first present the formula for expanding the second derivative of an even eigenfunction into a series of the even basis functions. After a lengthy calculation, using \textsc{Mathematica}'s algebraic manipulation capabilities as well as hyperbolic and trigonometric identities, from \eqref{eq:beta_series_even} we find, for a fixed $n\in\mathbb{N}$, and any $m=1,2,3,\ldots$\,, that
	{\small
	\begin{empheq}[left={\beta_{nm}^c=\empheqlbrace\,}]{align}
		&\frac{(\lambda_n^c)^2}{\left(\cos{\mu_m^c} - \cosh{(\sqrt{3} \mu_m^c)} \right) \left((\lambda_n^c)^6 - (\mu_m^c)^6\right) \left(\sqrt{3} \sin{\lambda_n^c} - \sinh{(\sqrt{3} \lambda_n^c)}\right) } \label{eq:sec_deriv_even_form}	\\
		&\times \Bigg\{2 \lambda_n^c (\mu_m^c)^4 \Big[ \cos{2\lambda_n^c} \left(\sqrt{3} \cos{2\mu_m^c} + 3 \sin{\mu_m^c} \sinh{(\sqrt{3}\mu_m^c)} - \sqrt{3} \cos{\mu_m^c} \cosh{(\sqrt{3} \mu_m^c)} \right) \nonumber  \\ 
		&\quad +3 \sin{\lambda_n^cf} \sinh{(\sqrt{3} \lambda_n^c)} \left(\cos{2 \mu_m^c} + \sqrt{3} \sin{\mu_m^c} \sinh{(\sqrt{3} \mu_m^c)} - \cos{\mu_m^c}\cosh{(\sqrt{3} \mu_m^c)}\right) \nonumber  \\
		&\quad +\cos{\lambda_n^c} \cosh{(\sqrt{3} \lambda_n^c)} \left(-\sqrt{3} \cos{2\mu_m^c} - 3\sin{\mu_m^c} \sinh{(\sqrt{3} \mu_m^c)} + \sqrt{3}\cos{\mu_m^c} \cosh{(\sqrt{3} \mu_m^c)}\right)\Big] \nonumber \\
		&\quad + 6\mu_m^5\sin{\mu_m^c} \left(\cos{\mu_m^c}  - \cosh{(\sqrt{3} \mu_m^c)}\right) \nonumber \\
		&\qquad \times \left(\sqrt{3}\sin{2\lambda_n^c} - 3 \cos{\lambda_n^c}  \sinh{(\sqrt{3} \lambda_n^c)} + \sqrt{3} \sin{\lambda_n^c}\cosh{(\sqrt{3} \lambda_n^c)}\right)\Bigg\}\,, \qquad \text {if}\ \  n\neq m\,,    \nonumber\\[4mm]
		&\frac{\lambda_n^c}{12\left(\cos{\lambda_n^c} - \cosh{(\sqrt{3} \lambda_n^c)} \right)  \left(\sqrt{3} \sin{\lambda_n^c} - \sinh{(\sqrt{3} \lambda_n^c)}\right) } \nonumber \\
		&\times \Biggl\{\sqrt{3} \left(7 \cos{4\lambda _n^c} - 31 \cos{2 \lambda _n^c}\right) + \sqrt{3} \cosh{(\sqrt{3} \lambda_n^c)} \left(31 \cos{\lambda_n^c} + \cos{3\lambda_n^c}\right)  \nonumber   \\
		&-\sqrt{3}\cosh{(2 \sqrt{3} \lambda_n^c)}  \left(1 + 7 \cos{2 \lambda_n^c}\right) - 3 \sinh{\sqrt{3} \lambda_n^c)} \left(\sin{3 \lambda_n^c} - 31 \sin{\lambda_n}\right) - 21 \sin{2\lambda_n^c} \sinh{(2 \sqrt{3} \lambda_n^c)} \nonumber  \\
		&+ 6\lambda_n^c \times\Biggl[  \sqrt{3} \sin{2 \lambda_n^c} + \sqrt{3} \cosh{(\sqrt{3} \lambda_n^c)}  \left(\sin{\lambda_n^c} - \sin{3 \lambda _n}\right)  - \sinh{( 2 \sqrt{3} \lambda _n)} \nonumber \\
		&- \left(\cos{3 \lambda_n^c} - 3 \cos{\lambda_n^c}\right)  \sinh{(\sqrt{3} \lambda_n^c)} \Biggr]\Biggr\}	\,, \qquad \qquad  \qquad \qquad  \qquad \qquad  \qquad \text {if}\ \  n=m\,. 	\nonumber
	\end{empheq}
	}	
The `$c$' case of the expansion~\eqref{eq:beta_series_even} and the coefficient formula \eqref{eq:sec_deriv_even_form} (for the second derivative's expansion back into the even eigenfunctions) were verified numerically for different values of $n$. In figure~\ref{fig:psi_5_even_sec_deriv}(a)  the profile of ${\rd^2\psi_5^{c}}/{\rd x^2}$ is plotted along that of its spectral approximation. A strong Gibbs effect is observed near the boundaries, which is expected since ${\rd^2\psi_5^{c}}/{\rd x^2}$ does not satisfy any of the boundary conditions~\eqref{eq:6th_BCs}. In figure~\ref{fig:psi_5_even_sec_deriv}(b), the convergence rate of the same spectral series~\eqref{eq:beta_series_even} is demonstrated. It is observed that ${|\beta_{5m}^c|}/{|c_m|}$ decays algebraically for large $m$ as ${|\beta_{5m}^c|}/{|c_m|}\sim m^{-0.8}$. This decay was confirmed for all the examined values of $n$. 

\begin{figure}
	\centering
	\includegraphics[width=0.9\textwidth]{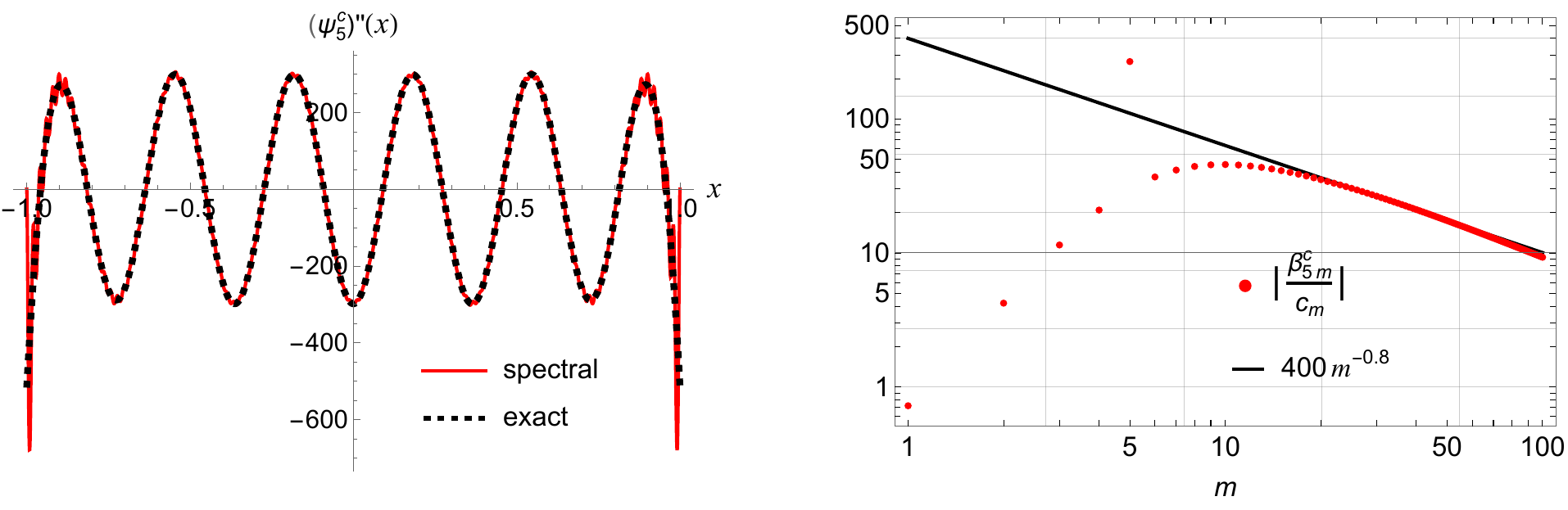}
	\text{ \quad (a) \qquad \qquad \quad \qquad \qquad \qquad  \qquad \qquad  \qquad  (b) \quad}
	\caption{The spectral expansion of ${\rd^2\psi_5^{c}}/{\rd x^2} \equiv (\psi^c_5)''(x)$ back into the even eigenfunctions~\eqref{eq:efuncs_6_015_even}, with $M=100$ terms. (a) The profile of the approximation compared to that of the exact representation of $(\psi^c_5)''(x)$. (b) The convergence rate of the absolute value of the spectral coefficients $|\beta_{5m}^c/c_m|$ alongside the asymptotic best fit  $400m^{-0.8}$ (for $m>20$). \label{fig:psi_5_even_sec_deriv}}
\end{figure}

For completeness, we also present the corresponding formula for expanding the second derivative of an odd eigenfunction into a series of the odd basis functions. Calculating the integral in~\eqref {eq:beta_series_odd}, we find, for a fixed $n\in\mathbb{N}$, and any $m=1,2,3,\ldots$\,, that
 {\small
 	\begin{empheq}[left={\beta_{nm}^s=\empheqlbrace\,}]{align}
 		&\ \,\frac{(\lambda^s_n)^2 \csch{(\frac{\sqrt{3}}{2} \lambda^s_n)} \sech{(\frac{\sqrt{3}}{2} \lambda^s_n)} }{4 \left(\cos{\mu^s _m} + \cosh{(\sqrt{3} \mu ^s_m)}\right) \left((\lambda^s_n)^6 - (\mu^s_m)^6\right)} \label{eq:sec_deriv_odd_form}	\\
 		&\ \,\times \Biggr\{ 2 (\lambda^s_n)^5 \sin{\lambda^s_n} \left(\cos{\lambda^s_n} - \cosh{(\sqrt{3}\lambda^s_n)}\right) \nonumber  \\ 
 		&\qquad \times  \left(\sqrt{3}\sin{2\mu^s_m} - 3 \cos{\mu^s_m} \sinh{(\sqrt{3}\mu^s_m)} + \sqrt{3} \sin{\mu^s_m}\cosh{(\sqrt{3}\mu^s_m)} \right)\nonumber  \\
 		&\quad +	2 \sqrt{3} (\mu^s_m)^2\,(\lambda^s_n)^3 \sin{\mu^s_m}  \sin{\lambda^s_n} \left(\cos{\mu^s_m}  - \cosh{(\sqrt{3}\mu^s_m)} \right) \left(\cos{\lambda^s_n} - \cosh{(\sqrt{3}\lambda^s_n)}\right) \nonumber \\
 		&\quad +12 (\mu^s_m)^4 \lambda^s_n \cos{\lambda^s_n} \sinh{(\sqrt{3}\lambda^s_n)} \left(\sin{2\mu^s_m} + \sqrt{3} \cos{\mu^s_m} \sinh{(\sqrt{3}\mu^s_m)} + \sin{\mu^s_m} \cosh{(\sqrt{3}\mu^s_m)}\right)\nonumber \\
 		&\quad +  3 (\mu^s_m)^5 \cos{\mu^s_m} \left(   \cos{\mu^s_m} + \cosh{(\sqrt{3}\mu^s_m)}  \right) \nonumber\\
 		&\qquad \times \left[\sqrt{3} \left(\cos{2 \lambda^s_n} + 3\right)  -  6 \sin{\lambda^s_n} \sinh{(\sqrt{3}\lambda^s_n)}  -   4 \sqrt{3} \cos{\lambda^s_n} \cosh{(\sqrt{3}\lambda^s_n)}   \right]   \Bigg\}\,, \qquad   \text {if}\ \  n\neq m\,,    \nonumber\\[4mm]
 		& 	  -\frac{\lambda^s_n \csch{(\frac{\sqrt{3}}{2} \lambda^s_n)}\sech{(\frac{\sqrt{3}}{2} \lambda^s_n)} }{96 \left( \cos{\lambda^s_n} + \cosh {(\sqrt{3}\lambda^s_n)} \right)}\nonumber \\
 		&\times \Biggl\{6 \left(\sin{\lambda^s_n} + \sin{3\lambda^s_n} \right) \sinh{(\sqrt{3}\lambda^s_n)} - 42 \sin{2\lambda^s_n} \sinh{(2\sqrt{3}\lambda^s_n)}  \nonumber   \\
 		& \quad +\sqrt{3} \left(8 \cos{2\lambda^s_n} + 13 \cos{4\lambda^s_n} + 3\right) + 22 \sqrt{3} \left(\cos{\lambda^s_n} - \cos{(\sqrt{3}\lambda^s_n)} \right) \cosh{(\sqrt{3}\lambda^s_n)}\nonumber  \\
 		& \quad + 24 \lambda^s_n \sinh{(\sqrt{3}\lambda^s_n)} \left(9 \cos{\lambda^s_n} + \cos{3\lambda^s_n} + 2 \cosh{(\sqrt{3}\lambda^s_n)} \right) \nonumber \\
 		&\quad  - 8 \sqrt{3} \left(2 \cos{2\lambda^s_n} + 1\right) \cosh{(2\sqrt{3}\lambda^s_n)}\Biggr\}	\,, \qquad \qquad  \qquad \qquad  \qquad \qquad  \qquad \text {if}\ \  n=m\,. 	\nonumber
 	\end{empheq}
 }

\subsection{Further expansion formulas}\label{sec:more_expansion_formulas}

For the first model problem to be considered in section~\ref{sec:model_problems}, the following expansion formulas for odd powers of $x$ is also required: 
\begin{equation}
	\label{eq:odd_powers_x_expansion}
		x^p = \sum_{m=0}^\infty \frac{1}{s_m} \chi^{\{p\}}_m \psi_{m}^s(x), \qquad\quad\ \ \chi^{\{p\}}_m = \Big\langle  x^p,\phi_m^s(x)\Big\rangle ,
\end{equation}
where $p \in \{ 1, 3, 5, 7\}$. The expressions for the coefficients $\chi^{\{p\}}_m$ for $m\in\mathbb{N}$ are given in the Appendix, specifically equation~\eqref{eq:odd_powers_x_coeff_exp}. The convergence rates of the spectral series~\eqref{eq:odd_powers_x_expansion} were found to be first-order algebraic, i.e., ${\D|\chi^{\{p\}}_m|}/{\D|s_m|}\sim m^{-1}$ for all $p\in\{1, 3, 5, 7 \}$.  The numerical results agree with the rough estimates obtained by inspecting the expansion coefficients~\eqref{eq:odd_powers_x_coeff_exp}, namely, recalling that $s_m=1$ for $m\geq7$ and noting the presence of the quantity $({\D 6 \cos{\lambda_m^s}})/{\D \lambda_m^s}$ as a dominant term in all of the expressions $\chi^{\{p\}}_m$, for  $p\in\{1, 3, 5, 7 \}$. As an indicative example, we present the expansion of $x^7$ along with the exact expression in figure~\ref{fig:x_seven_power}(a). As expected, a very strong Gibbs effect is observed at the boundaries since $x^7$ does not satisfy any of BCs~\eqref{eq:6th_BCs}.
\begin{figure}[h!]
	\centering
	\includegraphics[width=0.98\textwidth]{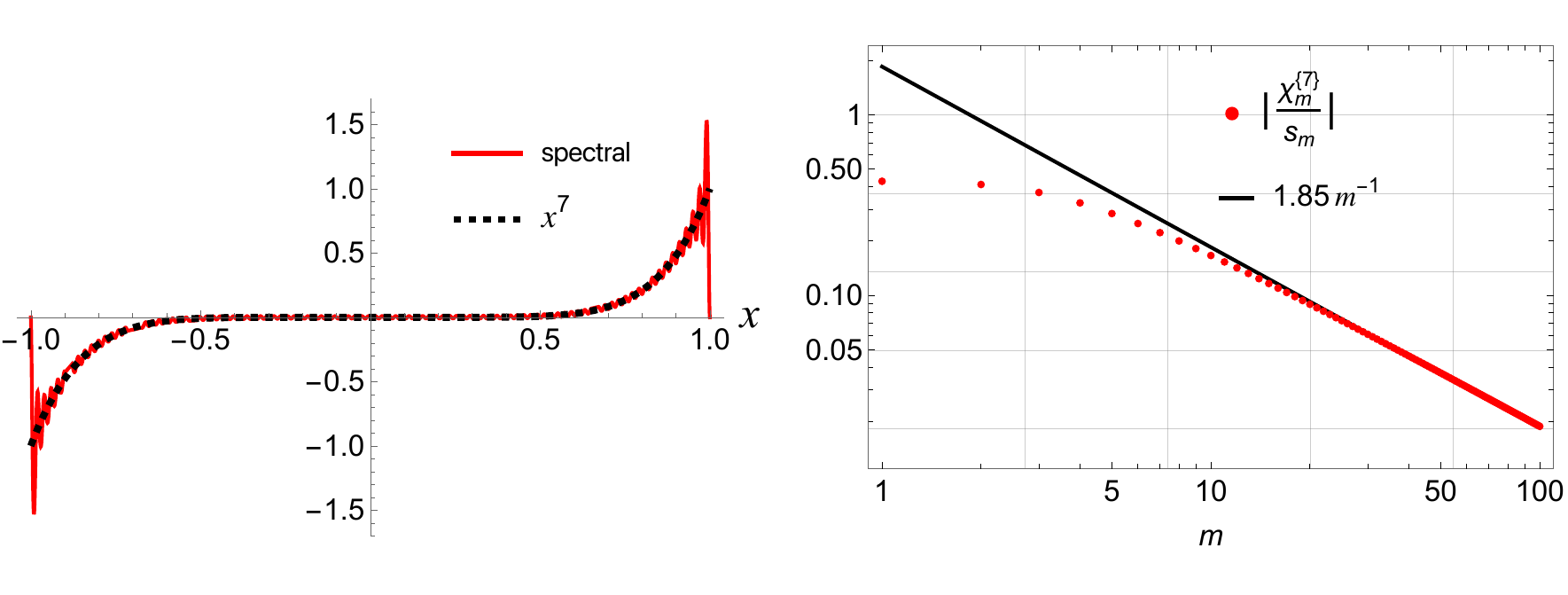}
	\text{ \quad (a) \qquad \qquad \quad \qquad \qquad \qquad  \qquad \qquad  \qquad \qquad  (b) \quad}
	\caption{The spectral expansion of $x^7$ into the odd eigenfunctions~\eqref{eq:efuncs_6_015_odd}, with $M=100$ terms. (a) The spectral approximation compared to $x^7$ itself. (b) Convergence rate of the absolute value of the spectral coefficients ${|\chi^{\{7\}}_m}/{s_m|}$  alongside the asymptotic best fit $1.85m^{-1}$  (for $m>20$). \label{fig:x_seven_power}}
\end{figure}

For the second model problem that we will consider in section~\ref{sec:model_problems} below, the expansion formula for trigonometric functions of the following form is also needed: 
\begin{equation}
	\label{eq:cos_k_expansion}
	\cos{k \pi x} = \sum_{m=1}^\infty \frac{1}{c_m} \ \hat{cs}^{\{k\}}_m \psi_{m}^c(x), \qquad\quad\ \   \hat{cs}^{\{k\}}_m = \Big\langle \cos{k \pi x} ,\phi_m^c(x)\Big\rangle  = \frac{6 (-1)^k (\lambda^c_m)^5 \sin{\lambda^c_m}}{(\lambda^c_m)^6-\pi ^6 k^6}\,,
\end{equation}
where $k\in\mathbb{Z}$. The coefficient $\hat{cs}^{\{k\}}_0 = 0$ for any $k\in\mathbb{Z}$. We have tested the formula using numerical computations, and as an indicative example, we present our results for $\cos{2\pi x}$, since it is the specific case that appears in the second model problem. These are depicted in figure~\ref{fig:Cos2PiX}. As before, a strong Gibbs effect is observed near the boundaries, and the convergence rate of the coefficients is first-order algebraic.
\begin{figure}[h!]
	\centering
	\includegraphics[width=0.98\textwidth]{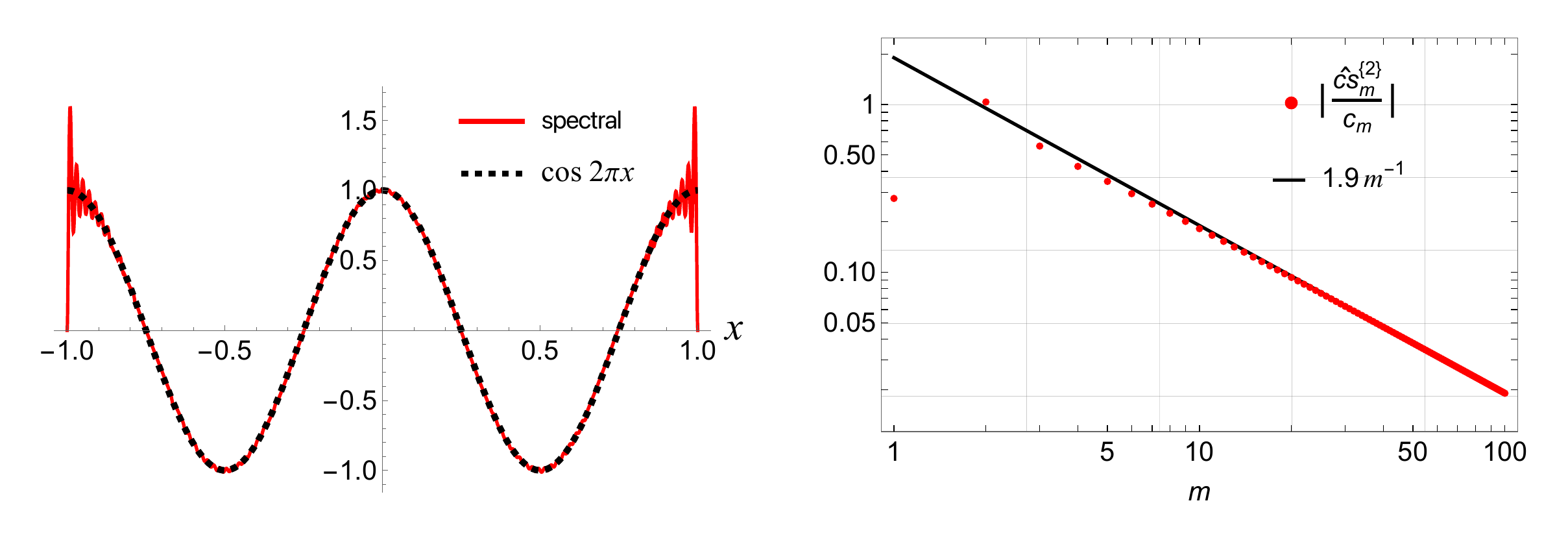}
	\text{ \quad (a) \qquad \qquad \quad \qquad \qquad \qquad  \qquad \qquad  \qquad \quad  (b) \quad}
	\caption{The spectral expansion of $\cos{2\pi x}$ back into the even eigenfunctions~\eqref{eq:efuncs_6_015_even}, with $M=100$ terms. (a) The spectral approximation compared to $\cos{2\pi x}$ itself. (b) Convergence rate of the absolute value of the spectral coefficients ${|\hat{cs}^{\{2\}}_m}/{c_m|}$  alongside the asymptotic best fit  $1.9m^{-1}$ (for $m\geq20$). \label{fig:Cos2PiX}}
\end{figure}

\section{Model problems: Results and discussion}\label{sec:model_problems}

To demonstrate the key features of the proposed Petrov--Galerkin spectral expansion for non-self-adjoint sixth-order-in-space parabolic-type equations, such as \eqref{eq:elastic_film_linear}, it suffices to consider steady problems, such that $\partial u/\partial t = 0$. 

In this section, motivated by our previous work on the self-adjoint case \cite{NectarIvan2023}, we use the \emph{method of manufactured solutions} \cite{Roache2002} to propose two model problems, with exact solutions, that demonstrate the spectral accuracy of the eigenfunction expansion.

\newpage

\subsection{Model problem I}

Model problem I is the BVP:
\begin{subequations}
	\label{eq:model_odd_1}
    \begin{align}\label{eq:model_odd_1:de}
    	\frac{\rd^6 u}{\rd x^6} &= f(x) ,\\
		\left.\frac{\rd u}{\rd x}\right|_{x=\pm1} = \left.\frac{\rd^2 u}{\rd x^2}\right|_{x=\pm1} = \left.\frac{\rd^5 u}{\rd x^5}\right|_{x=\pm1} &= 0 , \label{eq:model_odd_1:bc}
	\end{align}
where
	\begin{equation}
    	f(x) = -100\,800 x + 907\,200 x^3  - 1\,995\,840 x^5  + 1\,235\,520 x^7.
		\label{eq:model_odd_1:rhs}
	\end{equation}%
\end{subequations}
The model problem~\eqref{eq:model_odd_1} admits the exact solution:
\begin{equation}
	u(x) = u_\mathrm{exact}(x) = x(x-1)^6(x+1)^6.
	\label{eq:model_odd_1:exact}
\end{equation}
The exact solution~\eqref{eq:model_odd_1:exact} as well as the inhomogeneous term $f(x)$ in \eqref{eq:model_odd_1:de} are odd functions, and BCs~\eqref{eq:model_odd_1:bc} are symmetric. Thus, we need \emph{only} use the \emph{odd} eigenfunctions  $\Big\{\psi_n^s(x)\Big\}_{n=1}^{\infty}$ for the spectral expansion.		
		
To apply the Petrov--Galerkin spectral method developed in section~\ref{sec:Galerkin}, we introduce the truncated series 
\begin{equation}
	u(x) \approx u_{\mathrm{spectral}}(x) = \sum_{n=1}^{M} u_n^s \psi_n^s(x) .
	\label{eq:u_x_expansion_formula}
\end{equation}
Now, we substitute \eqref{eq:u_x_expansion_formula} into \eqref{eq:model_odd_1:de}, take inner products with $\Big\{\phi_{m}^s(x)\Big\}_{m=1}^{M}$\,, and employ the biorthogonality relations~\eqref{eq:biorth_six}, as in the derivation of \eqref{eq:dynamical_system}. 
Next, the coefficients $\big\{u_n^s\big\}_{n=1}^M$ are obtained as the solution of the \emph{diagonal} system of equations  
\begin{multline}
	\big[-s_n  (\lambda^s_{n})^6\big] \delta_{nm} u_{n}^s  =  -100\,800 \chi^{\{1\}}_{m} + 907\,200 \chi^{\{3\}}_{m} \\ 
	- 1\,995\,840 \chi^{\{5\}}_{m}  + 1\,235\,520 \chi^{\{7\}}_{m}\,,\qquad 	
	m=1,2,\ldots,M,
	\label{eq:model_1_um}
\end{multline}
where $\delta_{nm}$ is Kronecker's delta.
\begin{figure}[th!]
	\centering
	\includegraphics[width=0.98\textwidth]{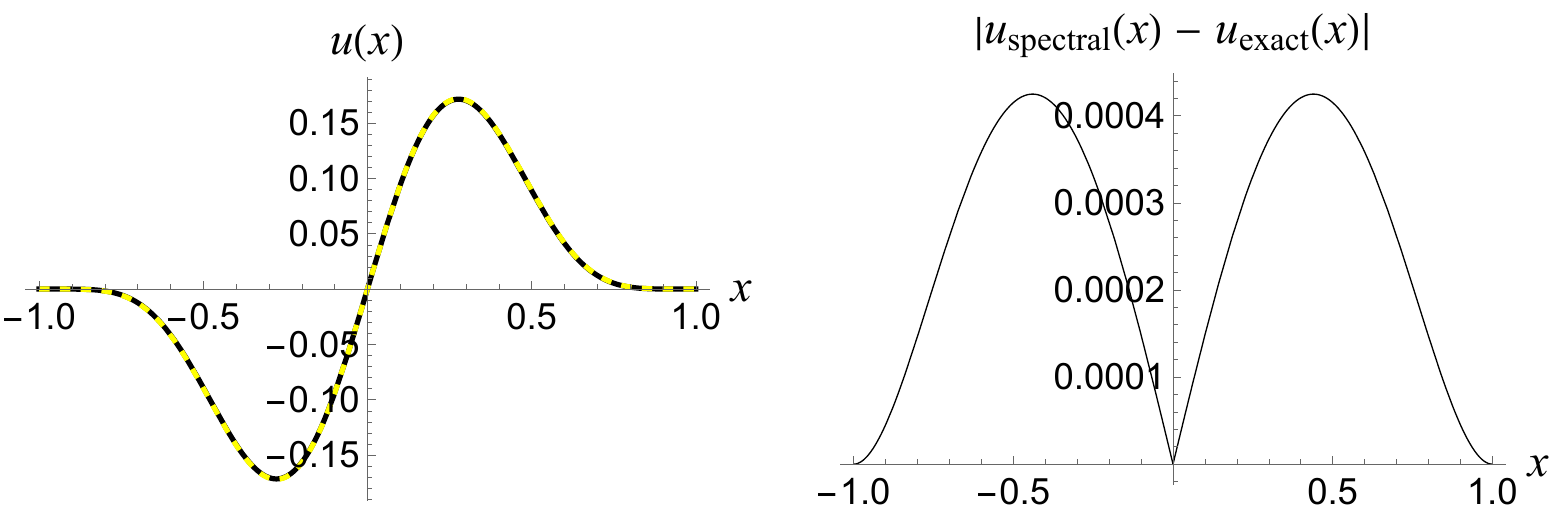}
	\text{ \quad (a) \qquad \qquad \quad \qquad \qquad \qquad  \qquad \qquad  \qquad \quad  (b) \quad}
	\caption{(a) The exact (black dashed) and spectral (yellow) solution profiles of model problem~\eqref{eq:model_odd_1}. The spectral series was computed with $M=100$ terms. (b) The absolute value of their difference $|u_{\rm spectral}(x)-u_{\rm exact}(x)|$ as a function of the spatial variable $x\in[-1,+1]$. \label{fig:model_prob_1_prof_error}}
\end{figure}

The exact solution~\eqref{eq:model_odd_1:exact} and the spectral approximation~\eqref{eq:u_x_expansion_formula}--\eqref{eq:model_1_um} using the odd sixth-order eigenfunctions are compared in figure~\ref{fig:model_prob_1_prof_error}(a), showing that they agree very well (at least visually). As seen in figure~\ref{fig:model_prob_1_prof_error}(b), the spectral expansion is  accurate within $5\times 10^{-4}$.

The convergence rate of the spectral series~\eqref{eq:u_x_expansion_formula} for problem~\eqref	{eq:model_odd_1} is shown in figure~\ref{fig:model_prob_1_coeffs}. It is important to note that the overall convergence rate of the spectral series is $\mathrm{O}(m^{-6.9})$, despite the fact that (i) the right-hand side (RHS) function~\eqref{eq:model_odd_1:rhs} does not satisfy the BCs~\eqref{eq:model_odd_1:bc} leading to a strong Gibbs effect, and (ii) the convergence rate of power terms forming the RHS is a very modest $\mathrm{O}(m^{-1})$.
\begin{figure}[h!]
	\centering
	\includegraphics[width=0.5\textwidth]{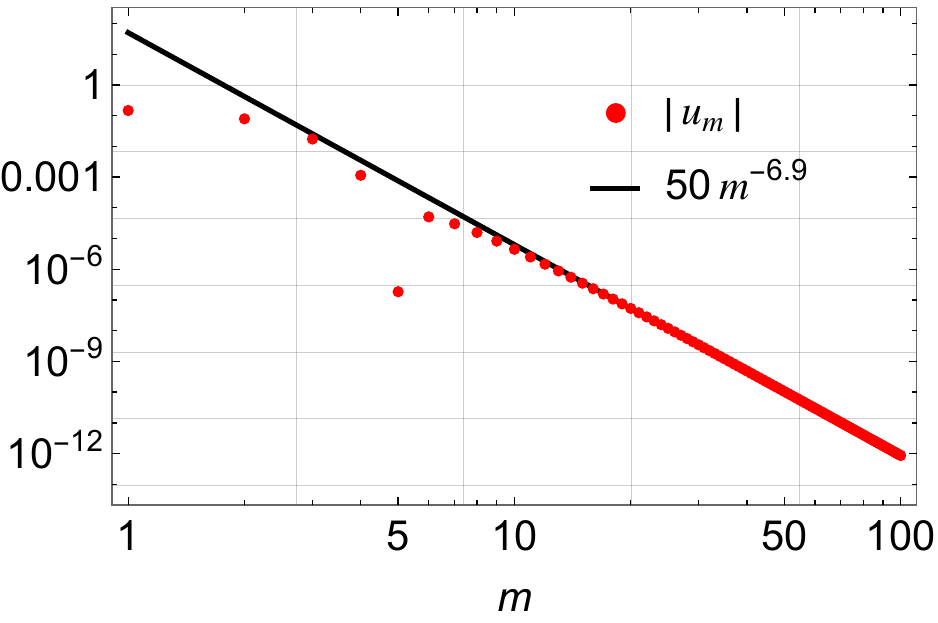}
	\caption{Log-log plot demonstrating the convergence rate of the coefficients of the approximate spectral series solution~\eqref{eq:u_x_expansion_formula} of model problem~\eqref{eq:model_odd_1}. Red dots: the absolute value of the spectral coefficients $|u_m|$.  Solid black line: The best fit $50m^{-6.9}$. The spectral series was truncated at $M=100$ terms. }
	\label{fig:model_prob_1_coeffs}
\end{figure}
The rapid convergence of the spectral approximation is an innate characteristic of our Petrov--Galerkin spectral method and the fact that the eigenfunctions $\psi^s_n(x)$ satisfy \eqref{eq:ODE_hg_6}. Specifically, due to the sixth derivative term, quantities that are proportional to $(\lambda_m^s)^6$ appear on the main diagonal of the coefficient matrix of the diagonal (algebraic) linear system~\eqref{eq:model_1_um} for the coefficients. Therefore, in effect, to obtain the sought coefficients, we divide each element of the RHS column vector by the corresponding sixth-order term $-s_m (\lambda_m^s)^6$, where $s_m$ is given by \eqref{eq:biorthog_odd}.  
We assert that this convergence rate indicates that the accuracy of the spectral solution would have been even higher, but for the very large coefficients of the RHS polynomial function~\eqref{eq:model_odd_1:rhs}. These large coefficients cannot be avoided when constructing a suitable sixth-order model problem that has an exact odd polynomial solution.

\subsection{Model problem II}

Model problem II is the BVP:
\begin{subequations}
	\label{eq:model_even_2}
	\begin{align}\label{eq:model_even_2:de}
		\frac{\rd^6 u}{\rd x^6} - 256\pi^4 \frac{\rd^2 u}{\rd x^2}&= f(x) ,\\
		\left.\frac{\rd u}{\rd x}\right|_{x=\pm1} = \left.\frac{\rd^2 u}{\rd x^2}\right|_{x=\pm1} = \left.\frac{\rd^5 u}{\rd x^5}\right|_{x=\pm1} &= 0 , \label{eq:model_even_2:bc}
	\end{align}
	where
	\begin{equation}
		f(x) = -960\pi^6\cos{2\pi x} .
		\label{eq:model_even_2:rhs}
	\end{equation}%
\end{subequations}
The model problem~\eqref{eq:model_even_2} admits the exact trigonometric solution:
\begin{equation}
	u(x) = u_\mathrm{exact}(x) = \cos{4\pi x} - \cos{2\pi x} .
	\label{eq:model_even_2:exact}
\end{equation}
The exact solution~\eqref{eq:model_even_2:exact} as well as the inhomogeneous term $f(x)$ in \eqref{eq:model_even_2:de} are \emph{even} functions, and BCs~\eqref{eq:model_even_2:bc} are symmetric. Thus, we need \emph{only} use the \emph{even} eigenfunctions  $\Big\{\psi_n^c(x)\Big\}_{n=0}^{\infty}$ for the spectral expansion.		

To apply the Petrov--Galerkin spectral method developed in section~\ref{sec:Galerkin}, we approximate the sought function $u(x)$ with the truncated spectral series 
\begin{equation}
	u(x) \approx u_{\mathrm{spectral}}(x) = \sum_{n=0}^{M} u_n^c \psi_n^c(x) ,
	\label{eq:u_even_expansion_formula}
\end{equation}
and expand the RHS of \eqref{eq:model_even_2:rhs} into even eigenfunctions $\Bigl\{\psi_n^c(x)\Bigr\}_{n=0}^M$ with the aid of formula~\eqref{eq:cos_k_expansion}, obtaining
\begin{equation}
		f(x) = -960\pi^6\cos{2\pi x} \approx  -960\pi^6\sum_{n=0}^{M} \frac{\D \hat{cs}^{\{2\}}_m }{\D c_m} \psi_n^c(x) .
	\label{eq:rhs_spectral_model_even_2}
\end{equation}
Next, we introduce \eqref{eq:u_even_expansion_formula} and \eqref{eq:rhs_spectral_model_even_2} into \eqref{eq:model_even_2:de}, take successive inner products with $\Bigl\{\phi_{m}^c(x)\Bigr\}_{m=0}^{M}$\,, and employ the biorthogonality relations~\eqref{eq:biorth_six}, to obtain the linear algebraic system
\begin{equation}
	\big[c_n  (\lambda^c_{n})^6\big] \delta_{nm} u_n^c  + 256\pi^4\sum_{n=1}^M\beta_{nm}^c u_n^c = 960\pi^6 \hat{cs}^{\{2\}}_m \,,\qquad 	
	m=1,2,\ldots,M,
	\label{eq:model_2_um}
\end{equation}
where $\beta_{nm}^c$ is as in~\eqref {eq:sec_deriv_even_form} and $\delta_{nm}$ is Kronecker's delta. Note that since $\beta_{n0}^c=0$ and  $\hat{cs}^{\{2\}}_0 = 0$ it follows that $u_0^c = 0$.  The remaining spectral coefficients $\big\{u_n^c\big\}_{n=1}^M$ are obtained by solving system~\eqref{eq:model_2_um}. The coefficient matrix of this system is full; however, it is relatively small in size ($100\times100$) and can therefore be solved using any standard direct method. The size of the system is dictated by the number of terms  $M$ in expansion~\eqref{eq:u_even_expansion_formula}. The reason why it is not necessary to exceed $M=100$ in \eqref{eq:u_even_expansion_formula} will become apparent in the discussion that follows. 
\begin{figure}[th!]
	\centering
	\includegraphics[width=0.98\textwidth]{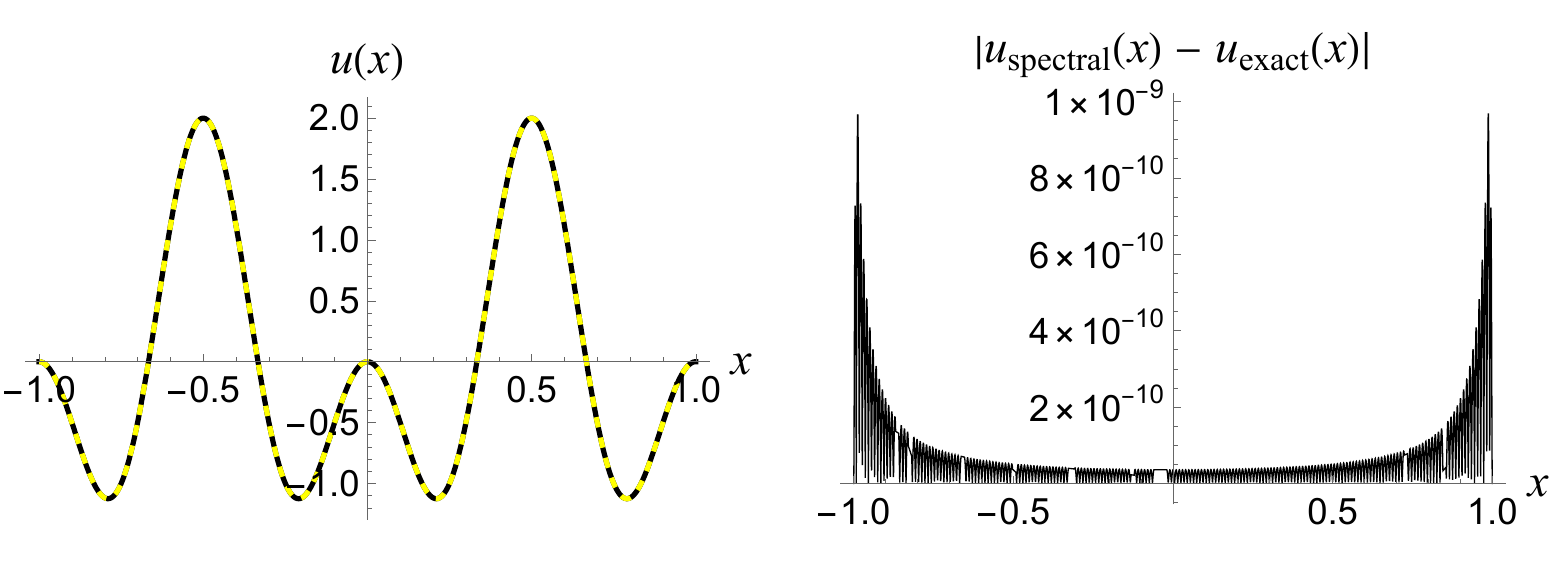}
	\text{ \quad (a) \qquad \qquad \quad \qquad \qquad \qquad  \qquad \qquad  \qquad \qquad  (b)}
	\caption{(a)  The exact (black dashed) and spectral (yellow) solution profiles of model problem~\eqref{eq:model_even_2}. The spectral series was computed with $M=100$ terms. (b) The absolute value of their difference $|u_{\rm spectral}(x)-u_{\rm exact}(x)|$ as a function of the spatial variable $x\in[-1,+1]$. \label{fig:model_prob_2_prof_error}}
\end{figure}

The profiles of the exact solution~\eqref{eq:model_even_2:exact} and the spectral approximation~\eqref{eq:u_even_expansion_formula} with \eqref{eq:model_2_um}, for $M=100$ terms are plotted together in figure~\ref{fig:model_prob_2_prof_error}(a). The two profiles are indistinguishable, which indicates an excellent agreement between the two solutions. Indeed, figure~\ref{fig:model_prob_2_prof_error}(b) clearly shows that the absolute error is less than $10^{-9}$. This example demonstrates the very high accuracy of the proposed method. 

The convergence rate of the spectral series~\eqref{eq:u_even_expansion_formula} is examined in figure~\ref{fig:model_prob_2_coeffs}, where the absolute values of the spectral coefficients $u_m$ are plotted along with the best fit curve $4500 m^{-6.9}$ using a log-log scale.  We note that this is the same convergence rate of O$(m^{-6.9})$ that was found in figure~\ref{fig:model_prob_1_coeffs} when solving model problem I. And this, despite the very strong Gibbs effect and slow convergence rates (of O$(m^{-0.8})$ and O$(m^{-1})$, respectively) observed in the spectral expansions of  $(\psi^c_m)''$  and $\cos{2\pi x}$.  We can therefore conclude that the convergence rates of the spectral series obtained using our Petrov--Galerkin spectral method based on the sixth-order eigenfunctions will be rapid, exceeding the expected O$(m^{-6})$ algebraic decay. 

\begin{figure}[h!]
	\centering
	\includegraphics[width=0.5\textwidth]{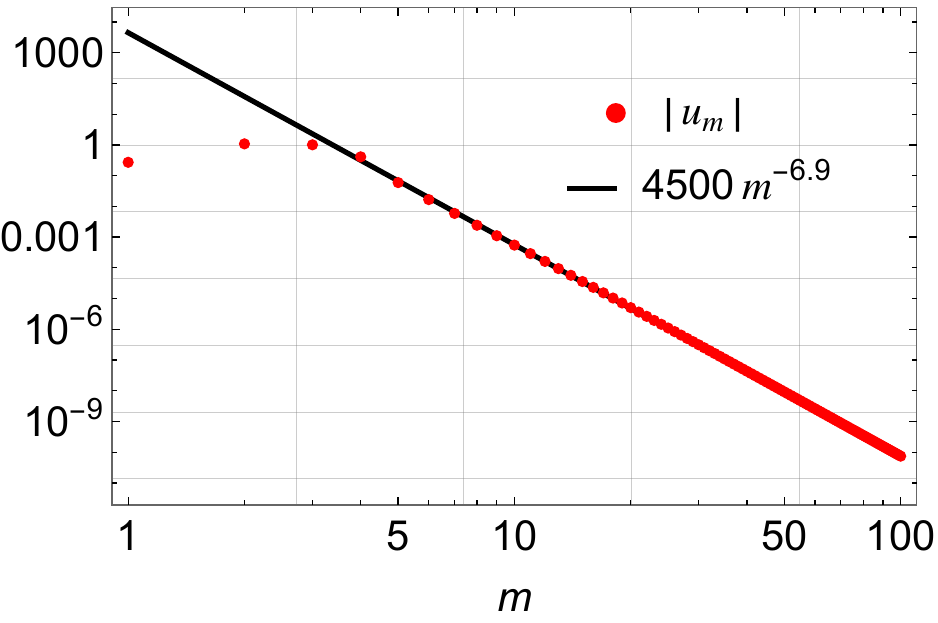}
	\caption{Log-log plot demonstrating the convergence rate of the coefficients of the approximate spectral series solution~\eqref{eq:u_even_expansion_formula} of model problem~\eqref{eq:model_even_2}. Red dots: the absolute value of the spectral coefficients $|u_m|$.  Solid black line: The best fit $4500m^{-6.9}$. The spectral series was truncated at $M=100$ terms. }
	\label{fig:model_prob_2_coeffs}
\end{figure}

\clearpage

\section{Conclusion}\label{sec:Conclusions}

We studied the non-self-adjoint \emph{sixth}-order boundary value problem arising from modeling the bending-dominated infinitesimal deformations and adjustment of an elastic-plated thin-film flow in a closed trough. Our study extends the previous investigation of Gabay \textit{et al}.~\cite{Gabay2023} on pinned films with surface tension in closed troughs, which led to a \emph{fourth}-order non-self-adjoint eigenvalue problem. Further, building on our earlier work \cite{NectarIvan2023} on the self-adjoint sixth-order case, we constructed the sixth-order biorthogonal eigenfunctions and, appealing to some of the classical results of Birkhoff \cite{Birkhoff1908}, we derived expansion formulas for various functions and the second derivatives of the eigenfunctions.

Further, we demonstrated the use of the biorthogonal eigenfunctions in a Petrov--Galerkin spectral method. Specifically, we applied the spectral method to two model problems (ordinary differential equations) with manufactured solutions. For each of the model problems, featuring polynomial and trigonometric functions in their solutions, respectively, we found that the convergence rate of the spectral series is rapid, exceeding the expected sixth-order algebraic rate (i.e., $|a_N| = \mathrm{O}(N^{-6})$, where $a_N$ are the coefficients of the expansion, and $N$ is the number of terms in the spectral series). Thus, we provided further numerical evidence for the validity of Birkhoff's \cite{Birkhoff1908} theory for higher-order non-self-adjoint eigenvalue problems.
    
In future work, it could be of interest to consider an unsteady problem and solve the dynamical system \eqref{eq:dynamical_system} numerically, as we did in our previous work \cite{NectarIvan2025}.



\ack ICC would like to acknowledge the US National Science Foundation, which supported his research on interfacial dynamics under grant CMMI-2029540, from which this work originated.

\appendix

\section*{Appendix}

\setcounter{section}{1}

For $m\ge1$, the expressions for the coefficients with in the expansion formulas~\eqref{eq:odd_powers_x_expansion} for the odd powers of $x$ are:
\begin{subequations}
	\label{eq:odd_powers_x_coeff_exp}
    \begin{align}
		\chi^{\{1\}}_m&= \frac{2\sqrt{3} \cos{\lambda_m^s} \sinh{(\sqrt{3} \lambda_m^s)}}{(\lambda_m^s)^2 \left(\cos{\lambda_m^s} +\cosh{(\sqrt{3} \lambda_m^s)}\right)}-\frac{6 \cos{\lambda_m^s}}{\lambda_m^s} , \\
		\chi^{\{3\}}_m&= \frac{6 \sqrt{3}  \cos{\lambda_m^s} \sinh{(\sqrt{3} \lambda_m^s)}}{(\lambda_m^s)^2 \left( 	\cos{\lambda_m^s} + \cosh{(\sqrt{3} \lambda_m^s)}\right)}-\frac{6 \cos{\lambda_m^s}}{\lambda_m^s} , \\
		\chi^{\{5\}}_m&= \frac{10 \sqrt{3} \left((\lambda_m^s)^4-24\right)  \cos{\lambda_m^s} \sinh{(\sqrt{3} 	\lambda_m^s)}}{(\lambda_m^s)^6 \left(\cos{\lambda_m^s} + \cosh{(\sqrt{3} \lambda_m^s)} \right)}-\frac{6  \cos{\lambda_m^s} }{\lambda_m^s} , \\
		\chi^{\{7\}}_m&= \frac{2 \cos{\lambda_m^s}}{{(\lambda_m^s)^8}} \left(\frac{7 \sqrt{3} \left((\lambda_m^s)^6 - 360 	(\lambda_m^s)^2 - 720\right) \sinh{(\sqrt{3} \lambda_m^s)}}{\cos{\lambda_m^s} +  \cosh{(\sqrt{3} \lambda_m^s)}} - 3 \lambda_m^s \left((\lambda_m^s)^6 - 5040\right)\right).
	\end{align}	
\end{subequations}


\clearpage

\section*{References}
\bibliography{Sixth_Order_clamped.bib}

\end{document}